\documentclass[a4paper,11pt]{amsart}

\usepackage{mathrsfs} 

\usepackage{amsthm}
\usepackage{hyperref}
\usepackage{amssymb}
\usepackage{latexsym}
\usepackage{amsmath}
\usepackage{amsfonts}
\usepackage[dvips]{graphicx}
\usepackage[utf8]{inputenc}
\usepackage{placeins}

\usepackage[T1]{fontenc}
\usepackage{url}
\usepackage{enumitem}
\usepackage{hyperref}
\usepackage{color}
\usepackage{bbm}
\hypersetup{colorlinks=true,citecolor=red}

\usepackage{mathtools}                                                      
\mathtoolsset{showonlyrefs=true}                                    

\usepackage[colorinlistoftodos,bordercolor=orange,backgroundcolor=orange!20,linecolor=orange,textsize=scriptsize]{todonotes}

\usepackage[deletedmarkup=xout]{changes}

\definecolor{darkblue}{rgb}{0,0,.5}

\usepackage{fancyhdr}
\newtheorem{Theorem}{Theorem}[part]
\newtheorem{Definition}{Definition}[part]
\newtheorem{Proposition}{Proposition}[part]

\newtheorem{Assumption}{Assumption}[part]
\newtheorem{Lemma}{Lemma}[part]
\newtheorem{Corollary}{Corollary}[part]
\newtheorem{Remark}{Remark}[part]

\makeatletter 
\@addtoreset{Definition}{section}

\@addtoreset{Theorem}{section}

\@addtoreset{Proposition}{section}

\@addtoreset{Property}{section}

\@addtoreset{Assumption}{section}

\@addtoreset{Corollary}{section}

\@addtoreset{Lemma}{section}

\@addtoreset{Remark}{section}

\@addtoreset{Example}{section}

\DeclareMathOperator{\spectre}{Sp}

\DeclareMathOperator{\Id}{Id}

\newcommand \Kfx{K_{f,x}}
\newcommand \Kfz{K_{f,z}}
\newcommand \Kbx{K_{b,x}}

\def \CA{C_A}

\newcommand{\bE}{\mathbb{E}}

\newcommand{\bL}{\mathbb{L}}

\newcommand{\bP}{\mathbb{P}}

\newcommand{\bR}{\mathbb{R}}

\newcommand{\cF}{\mathcal{F}}
\newcommand{\cG}{\mathcal{G}}

\newcommand{\cL}{\mathcal{L}}

\newcommand{\cN}{\mathcal{N}}

\newcommand{\sM}{\mathscr{M}}
\newcommand{\sS}{\mathscr{S}}

\def\1{\mathbf{1}}

\newcommand \normr[1]{\left\|#1\right\|_{\rho}}

\newcommand \E[1]{\bE\left[#1\right]}
\newcommand \Xx{X^x}

\newcommand \dd{{\rm d}}
\newcommand \dr{\dd r}
\newcommand \ds{\dd s}
\newcommand \dt{\dd t}
\newcommand \dx{\dd x}

\newcommand \dy{\dd y}
\newcommand \dW{\dd W}

\newcommand \normM[1]{\left\|#1\right\|}
\newcommand \normp[1]{\left|#1\right|_p}

\newcommand \norm[1]{\left|#1\right|}
\newcommand \Tr{{\rm Tr}}

\title[Numerical approximation of ergodic BSDEs]{Numerical approximation of ergodic BSDEs using  non linear Feynman-Kac formulas}

\author{Emmanuel Gobet}
\address{Centre de Math\'ematiques Appliqu\'ees (CMAP), CNRS, Ecole Polytechnique, Institut Polytechnique de Paris, Route de Saclay, 91128 Palaiseau Cedex, France}
\email{emmanuel.gobet@polytechnique.edu}

\author{Adrien Richou}
\address{Universit\'e de Bordeaux, Institut de Math\'ematiques de Bordeaux,
UMR CNRS 5251, 351 Cours de la Lib\'eration, 33405 Talence cedex, France.
}
\email{adrien.richou@math.u-bordeaux.fr}

\author{Lukasz Szpruch}
\address{School of Mathematics, University of Edinburgh, UK, and The Alan Turing Institute,
UK and Simtopia, UK}
\email{l.szpruch@ed.ac.uk}

\thanks{{The first author research has benefited from the support of the Chaire "Risques Financiers" of Fondation du Risque, and of the Chaire "Stress Test, RISK Management and Financial Steering" of the Ecole Polytechnique Foundation. The second author research has benefited from the support of the ANR Project ReLISCoP (ANR-21-CE40-0001).}}

\begin{document}

\begin{abstract}
In this work we study the numerical approximation of a class of ergodic Backward Stochastic Differential Equations. These equations are formulated in an infinite horizon framework and provide a probabilistic representation for elliptic Partial Differential Equations of ergodic type. In order to build our numerical scheme, we put forward a new representation of the PDE solution by using a classical probabilistic representation of the gradient. Then, based on this representation, we propose a fully implementable numerical scheme using a Picard iteration procedure, a grid space discretization and a Monte-Carlo approximation. Up to a limiting technical condition that guarantees the contraction of the Picard procedure, we obtain an upper bound for the numerical error. We also provide some numerical experiments that show the efficiency of this approach for small dimensions.
\\

\end{abstract}

\maketitle


{\sc MSC Classification (2020):} 65C30, 65C20, 65M12, 60H35.

\section{Introduction}
\subsection{Statement of the problem.}
{We study the numerical solution $(Y,Z,\lambda)$ of the Ergodic Backward Stochastic Differential Equation (EBSDE)
\begin{equation}
 \label{eq:EBSDE:intro}
 Y_t = Y_T + \int_t^T \left(f(X_s,Z_s)-\lambda\right) \ds -\int_t^T Z_s \dW_s,\quad 0\leqslant t \leqslant T,
\end{equation}
where the processes $(Y,Z)$ take values in some appropriate $\bL_2$ space, and $\lambda$ is a scalar (called \emph{ergodic cost}). Here $X$ is the solution of an ergodic forward SDE; detailed assumptions will be stated later, see Section \ref{section:model}. Our goal is to design a new numerical scheme for computing the solution of \eqref{eq:EBSDE:intro}, and to provide its error analysis with some numerical experiments to illustrate its performance.}

{Ergodic BSDE, introduced first by \cite{Fuhrman-Hu-Tessitore-07}, is an efficient tool to analyse optimal control problems with ergodic cost functionals; other methods are based on the Hamilton-Jacobi-Bellman equation, see for instance \cite{Arisawa-Lions-98} and \cite{Bensoussan-Frehse-02}. Let us highlight the link between EBSDE and stochastic control problem. It is now well known that, in a quite general setting, adjoint problems for stochastic control problems are given by solutions to BSDEs and their resolution gives access to the optimal control,  see \cite{Peng-93}-\cite{ma:yong:99}-\cite{zhan:17} for instance. Namely, consider the solution $(Y^T,Z^T)$ of the following BSDE, parameterized by $T>0$:
\begin{equation}
 \label{eq:BSDEfinitehorizon}
 Y_t^{T,x} = g(X_T^x)+\int_t^T f(X_s^x,Z^{T,x}_s) \ds -\int_t^T Z_s^{T,x} \dW_s,\quad 0\leqslant t \leqslant T.
\end{equation}
Under suitable assumptions (see \cite[Theorem 4.4]{Hu-Madec-Richou-14}) the following asymptotic expansion result holds: for some constants $L \in \bR$ and $C>0$,
\begin{align}\label{eq:BSDEfinitehorizon:erreur}
\norm{Y_0^{T,x} - \lambda T -Y_0^{x} -L} \leqslant C(1+|x|^3)e^{-T/C}
\end{align} 
where $Y^x_0$ is the solution of \eqref{eq:EBSDE:intro} for $X_0=x$.
This shows that solving \eqref{eq:EBSDE:intro} for any $x$ gives an explicit approximation for $Y_0^{T,x}$ in \eqref{eq:BSDEfinitehorizon} as $T$ is large.}

\subsection{State of the art.} 
{Although theoretical properties of EBSDEs have been well studied in the literature, see e.g. \cite{Fuhrman-Hu-Tessitore-07}, \cite{Richou-09}, \cite{Debussche-Hu-Tessitore-10}, \cite{Cohen-Hu-13}, \cite{Madec-14}, \cite{Guatteri-Tesitore-20}, 
{to the best of our knowledge, the recent article \cite{BrouxQuemerais-Kaakai-Matoussi-Sabbagh-24} is the only one that provides a numerical scheme dedicated to the resolution of EBSDEs. This scheme relies on a random horizon time approximation and a neural network space approximation. Our aim is to provide an alternative fully implementable scheme and to study the approximation error.}

About numerics for ergodic control, we refer to \cite{camp:legl:pard:00} and \cite{campillo2005numerical}  which approximate the stochastic control problem in infinite horizon using Markov chain approximations.

The literature about numerics for BSDEs is huge and it is mostly restricted to the case of finite horizon problem, in contrast with the infinite horizon setting of \eqref{eq:EBSDE:intro}. For the study of discretization errors under standard regularity conditions, see \cite{Zhang-04}; for singular terminal conditions, see \cite{Geiss-Geiss-Gobet-12}; for quadratic growth driver, see \cite{Chassagneux-Richou-13}. For an overview of numerical methods for BSDEs (in finite horizon), see the recent review \cite{ches:kawa:shin:yama:21}.
However, none of these works cover the case of infinite horizon BSDEs, except \cite{Beck-Gonon-Jentzen-20} who deals with infinite horizon BSDEs without dependency with respect to the $Z$-component {and, as already mentioned, the recent article \cite{BrouxQuemerais-Kaakai-Matoussi-Sabbagh-24} concerning EBSDEs.}}

\subsection{Our contributions and organization of the paper.}
{Our aim is therefore to design a first numerical scheme in multidimensional setting for solving EBSDE. For this, we establish a Markov representation of the value function and its gradient: for specialists, this is presumably not a surprizing result, but to the best of our knowledge, it was not  done so far. This is performed under the classical dissipativity assumption (Assumptions \ref{eq:generatorfandb:3}) and under a less usual but quite natural Hurwitz stability condition (Assumptions \ref{eq:generatorfandb':2}); see Proposition \ref{prop:existence:uniqueness:EBSDE}.
Then, 
we derive (see Theorem \ref{thm:representationBSDE:ergodic}) a fixed point equation to which the gradient is (the unique) solution: this equation writes as an expectation of functionals involving the required solution and sub-Gamma random variables. Some contraction properties of this fixed point equation are also investigated in Subsection \ref{subsec:contraction}. All these analytical results are the purpose of Section \ref{section:model}.}

{This probabilistic fixed point representation of the solution associated to a finite grid discretization of the space allows to design a suitable Picard iteration scheme in Section \ref{sec:numericalscheme}, for which we prove full convergence rates, with respect to the number of Picard iterations, the number of Monte-Carlo samples, the grid mesh on which the numerical solution is computed: see Proposition \ref{prop:numericalerror} and Corollary \ref{prop:errornum:ergodic}. 
On the technicalities side, we establish smoothness properties of the fixed-point mapping in suitable weighted norms, and we leverage  concentration-of-inequalities of Bernstein type (suitable for subGamma tails) to control uniformly statistical errors. 

Finally, some numerical experiments that illustrate theoretical convergences obtained are presented in Section \ref{sec:num:exp}.}

\subsection{Notations}
{In all this work, we consider a filtered probability space $(\Omega, \cF, (\cF_t)_{t\geq 0},\bP)$ which supports a $d$-dimensional Brownian motion $W=(W^1,\dots,W^d)^\top$. The filtration $(\cF_t)_{t\geq 0}$ is the one generated by $W$ augmented by the $\bP$-null sets, so that the filtration satisfies the "usual conditions".

\begin{description}
\item[Vector, matrix]
$\norm{x}$ denotes the Euclidean norm of $x\in \bR^d$; when $A$ is a matrix $\normM{A}$ stands for the matrix $2$-norm (i.e. subordinated to the Euclidean norm); $x^\top$ denotes the transpose of the vector $x$ and $\Tr(A)$ denotes the trace of $A$. For a vector $x \in \mathbb{R}^p$ (resp. a matrix $A \in \bR^{p \times q}$) and $R \in \mathbb{R}^+$, we denote $\lfloor x \rfloor_R$ {(resp. $\lfloor A \rfloor_R$)} the projection of $x$ (resp. $A$) on the Euclidean ball $\bar B(0,R)$ of $\bR^p$ (resp. $\bR^{p \times q}$).
\item[Function]
$C^0(\bR^p,\bR^q)$ denotes the set of functions $f : \bR^p \rightarrow \bR^q$ that are continuous.
$C^1(\bR^p,\bR^q)$ (resp. $C^1_b(\bR^p,\bR^q)$) denotes the set of functions $f : \bR^p \rightarrow \bR^q$ that are differentiable with a continuous (resp. continuous bounded) derivative. For a bounded function $f \in C^0(\bR^p,\bR^q)$ (resp. $C^0(\bR^p,\bR^{q\times r})$), we denote $|f|_{\infty}:=\sup_{x \in \mathbb{R}^p}|f(x)|$ (resp. $\|f\|_\infty:=\sup_{x \in \mathbb{R}^p}\|f(x)\|$).

For a function $f:=(f_i)_{1 \leqslant i \leqslant p} \in C^1_b(\bR^p,\bR^q)$, we denotes $\nabla_x f$ the function $\bR^p \owns x \mapsto (\partial_{x_j} f_i)_{i,j} \in \bR^{p,q}$. In particular, when $p=1$, $\nabla_x f$ is a row-vector valued function.
\item[Random variables and stochastic processes] 
For $p\geq 1$, $\bL_p$ denotes the set of (scalar or vector-valued) random variables $X$ with finite norm $\normp{X}:=(\E{\norm{X}^p})^{1/p}<+\infty$. $\bL_\infty$ stands for the set of essentially bounded random variables.

$\sS^2_T$ is the set of scalar adapted continuous processes $Y$ on $[0,T]$ such that 
$$\E{ \sup_{s \in [0,T]} |Y_s|^2}<+\infty.$$
$\sS^2_{loc}$ denotes the set of continuous processes $Y$ on $\bR^+$ such that  $(Y_t)_{t \leqslant T} \in \sS^2_T$, for all $T>0$.
$\sM^2_T$ is the set of $\bR^{1\times d}$-valued predictable processes $Z$ on $[0,T]$ such that
$$\E{ \int_0^T \|Z_s\|^2 \ds }<+\infty.$$
Observe that we write $Z_t$ as a row vector (in a coherent manner with writing the stochastic integral $\int_0^t Z_s \dW_s$).
$\sM^2_{loc}$ denotes the set of continuous processes $Z$ on $\bR^+$ such that $(Z_t)_{t \leqslant T} \in \sM^2_T$, for all $T>0$.

\item[Specific distributions] 
For $\ell>0$ and $a>0$, we denote $\Gamma(a,\ell)$ the gamma distribution with density (with respect to the Lebesgue measure)
$$x \mapsto \frac{\ell^a}{\Gamma(a)}x^{a-1}e^{-\ell x}\1_{(0,+\infty)}(x).$$
We recall the scaling property between distributions $\Gamma(a,\ell)\stackrel{d}=\ell^{-1}\Gamma(a,1)$, {and the special value of the Gamma function $\Gamma(\frac 12)=\sqrt \pi$.}

We denote by $F$ the cumulative distribution function of the Gaussian distribution $\mathcal{N}(0,1)$.
\end{description}}

\section{Analytical results}
\label{section:model}
{\subsection{Model and value function}
We consider the following ergodic BSDE
\begin{equation}
 \label{eq:EBSDE}
 Y_t = Y_T + \int_t^T \left(f(X_s,Z_s)-\lambda\right) \ds -\int_t^T Z_s \dW_s,\quad 0\leqslant t \leqslant T,
\end{equation}
where $(Y,Z,\lambda)$ is a solution in the space $\sS^2_{loc}\times \sM^2_{loc} \times \bR$ and $X$ is the solution of the forward $d$-dimensional SDE
\begin{align}
 \label{eq:SDE}
 X_t = x + \int_0^t b(X_s) \ds + \Sigma W_t, \quad 0 \leqslant t.
\end{align}
Existence and uniqueness of $(Y,Z,\lambda)$ solution to \eqref{eq:EBSDE} will be stated  in Proposition \ref{prop:existence:uniqueness:EBSDE}.
We assume following assumptions on $f$ and $b$.
\begin{Assumption}
 \label{eq:generatorfandb}
There exist constants $\Kfx\geqslant 0$, $\Kfz\geqslant 0$, $\Kbx\geqslant 0$ and $\eta>0$ such that, $\forall x,x' \in \bR^d, z,z' \in \bR^{1 \times d}$,
 \begin{enumerate}[labelindent=\parindent,leftmargin=*,label={\bf (A-\arabic*)}]
  \item  \label{eq:generatorfandb:1} $\norm{f(x,z)-f(x',z')} \leqslant \Kfx\norm{x-x'}+\Kfz\|z-z'\|,$
  \item  \label{eq:generatorfandb:2} $\norm{b(x)-b(x')} \leqslant \Kbx\norm{x-x'},$
  \item \label{eq:generatorfandb:3} $\langle b(x)-b(x'),x-x' \rangle \leqslant -\eta \norm{x-x'}^2,$
  \item \label{eq:generatorfandb:4} $\Sigma$ is invertible.
 \end{enumerate} 
\end{Assumption}

Since $b$ is a Lipschitz function, the SDE \eqref{eq:SDE} has a unique strong solution for any starting point $x$ at time $0$: whenever necessary to emphasize on the  $x$-dependence of the solution, we shall denote it by  $X^x$. 
Owing to the condition of  \ref{eq:generatorfandb:3}, the solution $X$ admits a unique invariant measure denoted $\nu$, see \cite[Chapter 4]{khas:12}.
}

{Let us emphasize that these assumptions can be {weakened} in several directions in order to study the well-posedness or some properties of EBSDEs. For example, some existence and uniqueness results under weaker dissipativity assumptions than \ref{eq:generatorfandb:3} are also available in \cite{Debussche-Hu-Tessitore-10,Hu-Madec-Richou-14}: in these papers, we are allowed to consider an extra bounded nonlinear term in the drift of $X$. It is also possible to relax the invertibility of $\Sigma$ or to consider a multiplicative noise: see e.g. \cite{Fuhrman-Hu-Tessitore-07,Richou-09,Guatteri-Tesitore-20}. 

For some reasons that will be explained after, we will restrict our study from Subsection \ref{sec:Feynman-Kac-representation} to the special case where $X$ is solution of an ergodic multidimensional Ornstein-Uhlenbeck process. Namely, we will assume that the drift $b(x)=-Ax$ with a Hurwitz matrix $-A$ {(see condition \ref{eq:generatorfandb':2})}, so that the dynamics of $X$ writes
\begin{equation}
\label{eq:OU}
X_t = x-\int_0^t AX_s \ds + \Sigma W_t.
\end{equation}
Choosing this model family is not so restrictive in practice because of the popularity of this model in applications.
In this special case, Assumptions \ref{eq:generatorfandb} are replaced by the following alternative assumptions.

\begin{Assumption}
 \label{eq:generatorfandb'}
 There exists a matrix $A \in \mathbb{R}^{d \times d}$ such that $b(x)=-Ax$ for all $x \in \mathbb{R}^d$, and there exist constants $\Kfx\geqslant 0$, $\Kfz\geqslant 0$ and $a>0$ such that, $\forall x,x' \in \bR^d, z,z' \in \bR^{1 \times d}$,
 \begin{enumerate}[labelindent=\parindent,leftmargin=*,label={\bf (A-\arabic*')}]
  \item  \label{eq:generatorfandb':1} $\norm{f(x,z)-f(x',z')} \leqslant \Kfx\norm{x-x'}+\Kfz\|z-z'\|,$
  \item  \label{eq:generatorfandb':2} $\spectre{A} \subset \{z \in \mathbb{C} | \Re(z)>a\}$.
  \item \label{eq:generatorfandb':3} $\Sigma$ is invertible.
 \end{enumerate} 
\end{Assumption}}

{Let us remark that, when $A$ is a symmetric matrix, Assumptions \ref{eq:generatorfandb} are
easily fulfilled as soon as Assumptions \ref{eq:generatorfandb'} are satisfied. Nevertheless this it is no longer true for a general Hurwitz matrix: for instance $A=\begin{pmatrix} 1 & -3 \\ 0 & 1\end{pmatrix}$ is Hurwitz but $x\cdot A x \ngeqslant 0$ for some $x$. So, the introduction of these alternative assumptions is justified by the fact that we do not want to restrict our study to symmetric matrices $A$. Finally, for any Hurwitz matrix $-A$, i.e. satisfying \ref{eq:generatorfandb':2} with $a>0$, we define $\CA$ as the smallest constant such that 
\begin{equation}
\label{eq:bound:expA}
    \normM{e^{-At}}\leqslant \CA e^{-at}, \qquad \forall t\geqslant 0.
\end{equation}
Let us remark that this constant always exists and $\CA \geqslant 1$, see e.g. \cite[Section 2.1]{gobe:she:16} for some explicit estimates on this constant. Moreover, $\CA=1$ as soon as $A$ is symmetric.}

{\begin{Proposition}
\label{prop:existence:uniqueness:EBSDE}
 Let us assume that Assumptions \ref{eq:generatorfandb} or \ref{eq:generatorfandb'} are in force. 
Then the ergodic BSDE \eqref{eq:EBSDE} has a solution $(Y,Z,\lambda)$ such that
\begin{equation}
 \label{eq:repMarkEBSDE}
  Y_t = u(X_t), \qquad Z_t = \bar{u}(X_t)
 \end{equation}
 for two measurable functions satisfying the
 growth
  \begin{align}\label{eq:prop:existence:uniqueness:EBSDE:bounds}
 |u(x)| \leqslant C(1+|x|),\qquad 
 {|\bar{u}(x)|}
 \leqslant C, \qquad \forall x\in \bR^d.
 \end{align}
Moreover, the solution 
$(Y,Z,\lambda)$  is unique (up to a constant for $Y$) in the class of Markovian solutions with previous growth.
\end{Proposition}
The justification that $u\in C^1$ and that $\bar u(.)=\nabla_x u(.)\Sigma$ will be established later in Theorem \ref{thm:representationBSDE:ergodic}.

\begin{proof}
Let us start by assuming Assumptions \ref{eq:generatorfandb}. We refer to \cite{Debussche-Hu-Tessitore-10,Hu-Madec-Richou-14} for the proof of the existence and uniqueness result as well as the Markovian representation. The growth of $u$ comes from Theorem 4.4 in  \cite{Fuhrman-Hu-Tessitore-07}. To be precise, it is assumed in \cite{Fuhrman-Hu-Tessitore-07} that $f(.,0)$ is bounded. Nevertheless, as mentioned in the proof of Lemma 3.12 in \cite{Hu-Madec-Richou-14}, results stay true when $f(.,0)$ has a linear growth. Now, let us explain why $\bar{u}$ is bounded. By using the proof of Theorem 4.4 in  \cite{Fuhrman-Hu-Tessitore-07}, we have that $u$ is a Lipschitz function. Now we remark that $(Y_t,Z_t)_{t \in [0,1]}$ is the solution of the finite time horizon BSDE
$$Y_t = u(X_1) + \int_t^1 (f(X_s,Z_s) -\lambda) \ds -\int_t^1 Z_s \dW_s, \quad t \in [0,1].$$
Then, classical estimates on $Z$ gives us that $Z$ is bounded. 

Now we tackle the alternative Assumptions \ref{eq:generatorfandb'}. Up to our knowledge, this framework is not directly covered by published results on EBSDEs. Nevertheless, some standard computations show that for all $x,x' \in \mathbb{R}^d$, $t \geqslant 0$,
we have $\Xx_t- X^{x'}_t=e^{-At} (x-x')$ and then 
\begin{equation}
\label{eq:ineg1}
|\Xx_t- X^{x'}_t|\leqslant \CA e^{-at} |x-x'|.
\end{equation}
Moreover, for any $B>0$, there exists a constant $C_B>0$ such that, for any progressively measurable process $\beta$ bounded by $B$, we have
\begin{equation}
\label{eq:ineg2}
    \sup_{t \geqslant 0} \mathbb{E}^{\mathbb{Q}}[|X_t^x|^2] \leqslant C_B(1+|x|^2),
\end{equation}
where $\mathbb{Q}$ is the Girsanov change of probability associated to the process $\beta$, i.e. $\frac{d\mathbb{Q}}{d\mathbb{P}}$ is given by the Doléans-Dade exponential of $\beta$. Then, by checking all the proofs in \cite{Fuhrman-Hu-Tessitore-07,Debussche-Hu-Tessitore-10,Hu-Madec-Richou-14}, we can remark that Assumption \ref{eq:generatorfandb:3} is used only to prove estimates like \eqref{eq:ineg1} and \eqref{eq:ineg2}. Then, all results stay true when we replace Assumption \ref{eq:generatorfandb} by \ref{eq:generatorfandb'}. This fact was already highlighted in \cite{Guatteri-Tesitore-20}: see their assumption (A6) as well as explanations that follow.
\qed
\end{proof}}

{Next, the BSDE \eqref{eq:EBSDE} gives a probabilistic representation of the following elliptic PDE
\begin{equation}
\label{eq:PDEergodic}
\cL u(x) + f(x,\nabla_x u(x)\Sigma)=\lambda,
\end{equation}
where $\cL$ denotes  the generator of the semi-group associated to the SDE \eqref{eq:SDE}.

\begin{Proposition}
\label{prop:viscosity:EBSDE}
 Under assumptions of Proposition \ref{prop:existence:uniqueness:EBSDE}, $u$ given by \eqref{eq:repMarkEBSDE} is a viscosity solution of \eqref{eq:PDEergodic}.
\end{Proposition}

\begin{proof}
As in \cite{Fuhrman-Hu-Tessitore-07,Debussche-Hu-Tessitore-10}, we can consider, for all $\alpha>0$, the infinite horizon BSDE
$$Y_t^{\alpha,x} = Y_T^{\alpha,x} + \int_t^T \left(f(X_s^x,Z_s^{\alpha,x}) - \alpha Y_s^{\alpha,x} \right)\ds - \int_s^T Z^{\alpha,x}_s \dW_s,\quad 0 \leqslant t \leqslant T,$$
and define ${u}^{\alpha}(x) := Y_0^{\alpha,x}$, $\bar{u}^{\alpha}(x) := Y_0^{\alpha,x}-Y_0^{\alpha,0}$ for all $x\in \bR^{d}$. $\bar{u}^{\alpha}$ is Lipschitz continuous and has a uniform linear growth: there exists $C$ that does not depend on $\alpha$ such that  $|\bar{u}^{\alpha}(x)|\leqslant C(1+|x|)$ for all $x\in \bR^d$. To be precise, it is assumed in \cite{Fuhrman-Hu-Tessitore-07,Debussche-Hu-Tessitore-10} that $f(.,0)$ is bounded. Nevertheless, as mentioned in the proof of Lemma 3.12 in \cite{Hu-Madec-Richou-14}, results stay true when $f(.,0)$ has a linear growth.
By standard arguments, see e.g. proof of Theorem 5.74 in \cite{pard:rasc:14}, $\bar{u}^{\alpha}$ is a viscosity solution of the elliptic PDE
$$\cL u(x) + f(x,\nabla_x u(x)\Sigma)= \alpha u(x) + \alpha u^{\alpha}(0),\quad x \in \bR^d.$$
According to \cite{Fuhrman-Hu-Tessitore-07}, there exists a sequence $(\alpha_n)_{n \in \mathbb{N}}$ such that $\alpha_n \searrow 0$, $\alpha_n u^{\alpha_n}(0) \rightarrow \lambda$ and $\bar{u}^{\alpha_n} \rightarrow u$ uniformly on $\bR^{d}$ when $n \rightarrow + \infty$. Then Remark 6.3 in \cite{Crandall-Ishii-Lions-92} gives us that $u$ is a viscosity solution of \eqref{eq:PDEergodic}.\qed \end{proof}}

{\subsection{Time-randomized Feynman-Kac representation}
\label{sec:Feynman-Kac-representation}
Now, our goal is to obtain a representation of $v(x):=\nabla_x u(x)$ as an expectation of a functional involving $ v(.)$ and the process $X$. 

\paragraph{\it Heuristics.}
We first explain informally the derivation, without taking much  care of precise assumptions, for the sake of emphasising ideas.
Writing the ergodic BSDE \eqref{eq:EBSDE} with $X^x$ and using the Markovian representation {of Proposition}  \eqref{prop:existence:uniqueness:EBSDE}, we get
\begin{equation}
 \label{eq:EBSDE:markov}
 u(x) = \E{u(X_T^x) +\int_0^T  \left(f(X_s^x,\bar u(X_s^x))  -\lambda\right)\ds}.
\end{equation}
By informally differentiating the above with respect to $x$ and using the Malliavin calculus integration by parts formula  (see \cite[Exercise 2.3.5, p.142]{nual:06}) to avoid  differentiating the $f$ term and having a $\nabla_x \bar{u}$ term, we obtain
\begin{equation}
 \label{eq:EBSDEdiff}
 v(x) = \E{v(X_T^x)\nabla_x X_T^x +\int_0^T  U^x_s f(X_s^x,\bar u(X_s^x))  \ds}
\end{equation}
where $U_s^x$ is the (raw vector valued) Malliavin weight given by 
\begin{align}
U^x_s = \frac{1}{s}\left(\int_0^s  (\Sigma^{-1}\nabla_x X^x_r)^\top \dW_r\right)^\top.\label{eq:malliavin:weight}
\end{align}
If we assume for the moment that $\bar{u} = \nabla_x u \Sigma$, we can replace $\bar{u}$ in \eqref{eq:EBSDEdiff} by $v\Sigma$. At first sight, $v(.)$  solves a nice fixed-point equation  \eqref{eq:EBSDEdiff} where the right-hand side is an expectation of a functional of $v$, of the path of $X^x$ and of its tangent process $\nabla_x X^x$. But a careful inspection shows that the terms inside the expectation have likely exploding polynomial moments as $T$ goes to $+\infty$. To see this, consider the simplest case of Ornstein-Uhlenbeck process, in dimension $d=1$,  with $\Sigma=1$ and $b(x)=-a x$ for a scalar parameter $a>0$: then $\nabla_x X^x_t=e^{-at}$ and 
\begin{align}
U^x_s = \frac{1}{s}\int_0^s  e^{-ar} \dW_r\stackrel{d}=\cN\left(0,\frac{1-e^{-2a s}}{ 2a s^2}\right).
\end{align}
Taking a bounded $f$ gives an estimation of $\bL_p$-norm as follows:
\begin{align}
\normp{\int_0^T  U^x_s f(X_s^x,\bar u(X_s^x) )  \ds}&\leq \int_0^T \normp{ U^x_s f(X_s^x,\bar u(X_s^x) )  } \ds\\
&\leq \int_0^T C_p \sqrt{\frac{1-e^{-2a s}}{ 2a s^2}} \|f\|_\infty \ds:
\end{align}
the upper integral converges at $s=0$ but diverges at $s=T$ when $T=+\infty$. It shows that the \emph{usual} Malliavin weight does not lead to finite estimates as $T\to+\infty$, {which is a major flaw in the future perspective of using a Picard iteration\footnote{it would require that  $\Kfz$ be small enough, with more stringent conditions than those of this work, see e.g. Proposition \ref{prop:contraction}.}}.}

\paragraph{\it Solution and final derivation.} 
{Actually, there is no uniqueness of such weights $U^x_s$, it is known that they coincide up to conditional expectation given $X^x_s$, see \cite[Section 2.1]{four:lion:01}. 
To overcome the issue of integrability at $T=+\infty$, we follow a slightly different path, using the likelihood ratio method 
\cite[Chapter VII, Section 3]{asmu:glyn:07}: denote by $p(0,x;s,.)$ the density of $X_s^x$, which exists owing to Girsanov arguments under the condition \ref{eq:generatorfandb:4}. Assuming for a while appropriate smoothness with respect to $x$, we get that the first derivative of the integral term in \eqref{eq:EBSDE:markov} equals 
\begin{align*}
\nabla_x \E{\int_0^T f(X_s^x,\bar{u}(X_s^x))\ds}
&=\int_0^T \int_{\bR^d} f(x',\bar{u}(x'))\nabla_x p(0,x;s,x') \dx'\ds\\
&=\E{\int_0^T  f(X_s^x,\bar{u}(X_s^x)) \bar{U}^x_s \ds}
\end{align*}
where 
\begin{align}\bar{U}^x_s=\nabla_x (\log(p(0,x;s,x'))\big|_{x'=X_s^x}.\label{eq:formula:barU}
\end{align}}
{In Theorem \ref{thm:representationBSDE:ergodic} below, we prove that 
\begin{align}
\mathbb{E}\left[\int_0^{\infty} \left|  \bar U^x_s f(X_s^x,v(X_s^x) \Sigma) \right| \ds\right]  <+\infty
\end{align}
as a difference with the previous weight $U^x_s$. Actually whether the above holds in full generality is, so far, an open question (because of untractable formulas for $\bar{U}^x_s$), however we have established the required property in the subclass of ergodic models described by multidimensional Ornstein-Uhlenbeck processes (Assumptions \ref{eq:generatorfandb'}). Standard computations from \eqref{eq:OU} show that 
\begin{equation}
\label{eq:OU:bis}
\Xx_t=e^{-At} x+e^{-At} \int_0^t e^{As}\Sigma\dd W_s,
\end{equation}
hence $\Xx_t$ is distributed as a Gaussian vector, with mean $e^{-At }x$ and covariance
\begin{align}
\label{eq:bound:def}
\Sigma_t &:= \int_0^t e^{A(s-t)}\Sigma\Sigma^\top e^{A^\top(s-t)}\ds=
 \int_0^t e^{- Ar }\Sigma\Sigma^\top e^{-A^\top r}\dr.\end{align}
The matrix $\Sigma_t$ clearly inherits from the invertibility of $\Sigma$ (condition \ref{eq:generatorfandb:4}); in addition we have the following estimates
\begin{align}
\label{eq:bound:Sigma}
\begin{split}
&\Sigma_t \underset{t\to 0^+}{\sim} t\ \Sigma \Sigma^\top,\qquad\qquad\qquad
\Sigma_t \underset{t\to +\infty}{\sim} \Sigma_\infty,\\
&\normM{\Sigma_t^{-1}} \leqslant C(1\lor t^{-1})
,\qquad\qquad
{\normM{\Sigma_t} \leqslant \CA^2\frac{ \normM{\Sigma \Sigma^\top}}{2a},}\qquad \forall t\geqslant 0,
\end{split}
\end{align}
where $\Sigma_\infty$ is the covariance matrix of the invariant law $\nu$ given by
\begin{align}\label{eq:invariant:nu}
\nu := \cN\left(0,\int_0^{+\infty} e^{-Ar}\Sigma \Sigma^\top e^{-A^\top r} \dr\right).
\end{align}}
{In view of \eqref{eq:formula:barU} and the above Gaussian distribution, we easily get 
\begin{align}
\label{eq:barU:tildeU}
\begin{cases}
  \bar{U}^x_s &=  {(\Xx_s-e^{-A s} x )^\top\Sigma_s^{-1} e^{-A s}}
  =:   e^{-a s} \tilde U_s,\\[2mm]
    \tilde{U}_s &= e^{as}{(\Xx_s-e^{-A s} x )^\top\Sigma_s^{-1} e^{-A s}}= e^{as}  \left(\int_0^s e^{Au}\Sigma\dd W_u\right)^{\top} (e^{-As})^{\top} \Sigma_s^{-1} e^{-A s}. 
    \end{cases}
\end{align}
Observe that $(\tilde{U}_s)_{s >0}$ does not depend on $x$. {
Moreover, owing to \eqref{eq:OU:bis} and  \eqref{eq:bound:Sigma}, there is a time-uniform constant $C$ such that
\begin{equation}
    \label{eq:bound:acc:X}
    \E{\norm{X_s^x-e^{-A s}x}^2}\leqslant C(s\wedge 1),
    \end{equation}
which in turn implies (again using \eqref{eq:bound:Sigma} and \eqref{eq:bound:expA}) the following upper-bound: 
\begin{equation}
    \label{eq:bound:esp:Utilde}
    \E{|\tilde{U}_s|^2} \leqslant C(1 \vee s^{-1}),\quad \forall s>0,
\end{equation}
for a  new time-uniform constant $C$.}
This informal derivation leads to the next statement, which rigorous proof is postponed to Section \ref{proof:thm:representationBSDE:ergodic}. Note that in comparison with Proposition \ref{prop:existence:uniqueness:EBSDE}, the link between the value function $u$ and its supposedly derivative $\bar u$ is established, as well their continuities.}
{\begin{Theorem}
 \label{thm:representationBSDE:ergodic}
 Let us assume that Assumptions \ref{eq:generatorfandb'} are in force. 
 Then 
   \begin{enumerate} 
 \item \label{thm:representationBSDE:ergodic:1} $u \in C^1(\bR^d)$,
 \item \label{thm:representationBSDE:ergodic:2} $Z_t = v(X_t)\Sigma$ with $v:=\nabla_x u$, $\|v\|_{\infty}<+\infty$,
 \item \label{thm:representationBSDE:ergodic:3} the gradient $v$ is solution of the four following equations
 \begin{align}
 \label{eq:representationZ:finiteT}
 v(x) &= \E{ v(X^x_T)e^{-AT}+ \int_0^{T} e^{-as}\tilde{U}_s f(\Xx_s,{v(\Xx_s) \Sigma})\ds} \\
 \label{eq:representationZ:infiniteT}
 &=  \E{ \int_0^{+\infty} e^{-as}\tilde{U}_s f(\Xx_s,{v(\Xx_s) \Sigma})\ds} \\
  \label{eq:representationZ:finiteT:randomise}
   &= \E{v(X^x_T)e^{-AT}+\1_{G\leq T \theta }\frac{\sqrt{\pi}}{\theta}\sqrt{G}e^{-(\frac a\theta-1)G}\tilde{U}_{\frac{G}\theta} f\left(\Xx_{\frac{G}\theta},{v(\Xx_{\frac{G}\theta}) \Sigma}\right)}\\
 &= \frac{\sqrt{\pi}}{\theta}\E{\sqrt{G}e^{-(\frac a\theta-1)G}\tilde{U}_{\frac{G}\theta} f\left(\Xx_{\frac{G}\theta},{v(\Xx_{\frac{G}\theta}) \Sigma}\right)}
 \label{eq:representationZ:infiniteT:randomise}
\end{align}
where {$\tilde{U}_s$ given by \eqref{eq:barU:tildeU} satisfies \eqref{eq:bound:esp:Utilde}},
$\theta \in (0,a)$ and $G\overset d=\cG(1/2,1)$ is independent of $W$.
  \end{enumerate}
  \end{Theorem}}
  
  {\begin{Remark}[Application to the approximation of BSDE in large horizon]
  As recalled in introduction, a BSDE with driver independent of $Y$, such as \eqref{eq:BSDEfinitehorizon}, can be well approximated, as the horizon $T$ is large, by an EBSDE with the formula
  \begin{align}
  \label{eq:expansion:BSDE:long:H}
Y_0^{T,x} \approx \lambda T + Y_0^{x} + L
\end{align} 
with the error bound \eqref{eq:BSDEfinitehorizon:erreur}. Since $Y_0^{x}$ is defined {up} to a constant, the constant $L$ depends implicitly on the choice of the constant for $Y_0^{x}$. Once $v$ is obtained from Theorem \ref{thm:representationBSDE:ergodic}, one can deduce \begin{equation}
\label{eq:lambda}
 \lambda = \int_{\bR^d} f(x,{v(x) \Sigma})\nu(\dx)
\end{equation}
with $\nu$ the invariant probability measure \eqref{eq:invariant:nu}: indeed, we just have to integrate \eqref{eq:EBSDE:markov} with respect to $\nu$ and  apply Fubini theorem. Second, since $u$ is the antiderivative of $v$ up to constant, we can set
\begin{equation}
 \label{eq:v:integrated}
 Y_0^{x} = u(x) = \int_0^1 v(tx)x \dt,\quad \forall x \in \bR^d.
\end{equation} These arguments set the first two terms on the right hand side of \eqref{eq:expansion:BSDE:long:H}. 
The tuning of $L$ is more delicate. In view of \eqref{eq:BSDEfinitehorizon:erreur} and since we take $Y_0^{x=0}=0$, we have
$$L=\lim_{T\rightarrow +\infty} (Y_0^{T,x=0}-\lambda T),$$
with an exponential convergence.  A naive approach would consist in estimating $Y_0^{T,x=0}$ (using a usual numerical method for BSDE) for a few $T$, to get an estimation of $L$. The experiments related of this approach are postponed to further research.
\end{Remark}}

{\begin{Remark}
Let us remark that the invertibility of $\Sigma$ is not necessary to get the existence and uniqueness result of Proposition \ref{prop:existence:uniqueness:EBSDE}, see e.g. \cite{Fuhrman-Hu-Tessitore-07}.
Moreover, it is well known that the invertibility of $\Sigma$ is not necessary to get the invertibility of $\Sigma_t$ for all $t>0$. Indeed, Kalman condition on $A$ and $\Sigma$, i.e.
 \begin{equation*}
  \label{hyp:Kalman}
  \text{the matrix } (\Sigma | A\Sigma |...|A^{d-1}\Sigma) \text{ has rank } d,
 \end{equation*}
 is equivalent to the invertibility of $\Sigma_t$ for all $t>0$ (see \cite[Proposition 6.5, Chapter 5]{kara:shre:91}). Under this weaker assumption, $\tilde{U}_t$ remains well-defined for all $t>0$. Nevertheless, if $\Sigma$ is not invertible, its time singularity close to $0^+$ is too strong and $(\mathbb{E}[|U_t^x|])_{t \in (0,1]}$ becomes not integrable. Then, the invertibility of $\Sigma$ becomes necessary to get the Feynman-Kac representation of $v$ in Theorem \ref{thm:representationBSDE:ergodic}. 
\end{Remark}}

\subsection{Contraction properties of the Fixed point equation}
\label{subsec:contraction}
We have established that $v$ solves equations \eqref{eq:representationZ:finiteT:randomise} and \eqref{eq:representationZ:infiniteT:randomise} that can be seen as some fixed point equations.

More precisely, we define for all $T \in \mathbb{R}^+ \cup \{+\infty\}$, a map $\Phi_T$ such that, for all measurable function $w : \mathbb{R}^d \rightarrow \mathbb{R}^{1 \times d}$, $\Phi_T(w)$ is a measurable function from $\mathbb{R}^d$ to $\mathbb{R}^{1 \times d}$ given by, for all $x \in \mathbb{R}^d$,
\begin{equation}
\label{deq:def:PhiT}
\Phi_T(w)(x)= \E{w(X^x_T)e^{-AT}\1_{T<+\infty}+\1_{G\leq T \theta }\frac{\sqrt{\pi}}{\theta}\sqrt{G}e^{-(\frac a\theta-1)G}\tilde{U}_{\frac{G}\theta} f\left(\Xx_{\frac{G}\theta},{ w(\Xx_{\frac{G}\theta}) \Sigma}\right)}.
\end{equation}
Then, equations \eqref{eq:representationZ:finiteT:randomise} and \eqref{eq:representationZ:infiniteT:randomise} 
rewrite as 
$$\Phi_T(v) = v, \quad \forall \,\, T \in \mathbb{R}^+ \cup \{+\infty\}.$$
As a preparation to discuss numerical approximation schemes, we study the contraction property of $\Phi_T$. In order to do it, we firstly have to set a suitable norm on the space $C^0(\mathbb{R}^d,\mathbb{R}^{1 \times d})$. Let $\rho:\bR^d\mapsto [1,+\infty)$ a positive weight function. We assume that $\rho$ is an increasing function with respect to $|x|$ with a growth at most exponential and at least affine: there exists $C>0$ such that
\begin{equation}
    \label{eq:growth:rho}
    1+C^{-1}|x| \leqslant \rho(x) \leqslant Ce^{C|x|}, \quad \forall \, x \in \mathbb{R}^d.
\end{equation}
Some standard choices will correspond to polynomial or exponential weighting, i.e. $\rho(x)=(1+\alpha|x|)^\beta$ or $\rho(x)=\exp(\alpha |x|)$ for some parameter $\alpha > 0$ and $\beta\geqslant 1$. The $\rho$-norm of a function $u:\bR^d\mapsto \bR^{1\times d}$ is defined by
\begin{align}
\label{eq:rho:norm}
\normr{u}=\sup_{x\in \bR^d}\frac{|u(x)|}{\rho(x) },
\end{align}
and we denote $C^0_{\rho}(\mathbb{R}^d,\mathbb{R}^{1\times d})$ the Banach space of functions $w \in C^0(\mathbb{R}^d,\mathbb{R}^{1 \times d})$ such that $\normr{w}<+\infty$.

\begin{Proposition} 
\label{pr:Lpestimate}
Let us assume that Assumption \ref{eq:generatorfandb'} is fulfilled and 
\begin{align}
\label{eq:condition:normA:rho:1}
c_{T,\eqref{eq:condition:normA:rho:1}}&:=\sup_{x\in \bR^d}\E{\frac{\rho(X_T^x)}{\rho(x)}\1_{T<+\infty}}<+\infty,\\
\label{eq:condition:normA:rho:2}
c_{T,\eqref{eq:condition:normA:rho:2}}&:=\int_0^T e^{-as}\sup_{x\in \bR^d}\E{|\Sigma_s^{-1}(X_s^x-e^{-As} x)|\frac{\rho(X_s^x)}{\rho(x)}}\ds <+\infty.
\end{align}
Then, for any functions $w_1,w_2 \in C^0_{\rho}(\mathbb{R}^d,\mathbb{R}^{1\times d})$, we have $\Phi_T(w_1),\Phi_T(w_2) \in C^0_{\rho}(\mathbb{R}^d,\mathbb{R}^{1\times d})$ and 
$$\normr{\Phi_T(w_1) -\Phi_T(w_2)} \leqslant \kappa_T \normr{w_1-w_2},$$
with
\begin{equation}
    \label{eq:def:kappa}
    \kappa_T := \CA e^{-aT} c_{T,\eqref{eq:condition:normA:rho:1}}\1_{T<+\infty} + \CA K_{f,z} \|\Sigma\| c_{T,\eqref{eq:condition:normA:rho:2}}.
\end{equation}
In particular, if $\kappa_T<1$, $v^0 \in C^0_{\rho}(\mathbb{R}^d,\mathbb{R}^{1\times d})$ and $v^{n+1} = \Phi_T(v^n)$ for all $n\in \mathbb{N}$, then $v$ is the unique fixed point of $\Phi_T$ and it satisfies 
\begin{equation}
    \normr{v^n-v} \leqslant (\kappa_T)^n \normr{v^0 -v}.
\end{equation}
\end{Proposition}
The proof of Proposition \ref{pr:Lpestimate} is postponed to Section \ref{proof:pr:Lpestimate}.

Now we can {specify a little bit the condition  $\kappa_T<1$ according to the values of the constants
$c_{T,\eqref{eq:condition:normA:rho:1}}$,  $c_{T,\eqref{eq:condition:normA:rho:2}}$, $\CA$,  $c_{1,\eqref{eq:constant:growth:sigma-1}}$ and $c_{2,\eqref{eq:constant:growth:sigma-1}}$ that satisfy}
\begin{align}
    \label{eq:constant:growth:sigma-1}
    \|\Sigma_s^{-1}\|^{1/2} \leqslant c_{1,\eqref{eq:constant:growth:sigma-1}} + \frac{c_{2,\eqref{eq:constant:growth:sigma-1}}}{\sqrt{s}}, \quad \forall s>0.
\end{align}
Let us remark that the existence of these constants comes from \eqref{eq:bound:Sigma}.
\begin{Proposition}
\label{prop:contraction}

\begin{itemize}
    \item If $\CA=1$ and $\rho(x) = e^{\alpha |x|}$ with $\alpha>0$, then 
    \begin{align}
        \inf_{T \in \mathbb{R}^+ \cup\{+\infty\}}\kappa_T \leqslant K_{f,z} \|\Sigma\| \sqrt{d}\left(2e^{{\frac{\alpha^2 \|\Sigma \Sigma^\top\|}{a^2}}}F\left(\frac{{\sqrt 2}\alpha \|\Sigma \Sigma^\top\|^{1/2}}{{\sqrt a}}\right)\right)^{d/2}  \left(\frac{c_{1,\eqref{eq:constant:growth:sigma-1}}}{a} + \frac{\sqrt{\pi}c_{2,\eqref{eq:constant:growth:sigma-1}}}{\sqrt{a}}\right)
    \end{align}
    and this upper-bound is an upper-bound for $\kappa_{\infty}$. In particular, $\kappa_{\infty}<1$ as soon as $\alpha$ is small enough and 
    $$K_{f,z} \|\Sigma\| \sqrt{d} \left(\frac{c_{1,\eqref{eq:constant:growth:sigma-1}}}{a} + \frac{\sqrt{\pi}c_{2,\eqref{eq:constant:growth:sigma-1}}}{\sqrt{a}}\right)<1.$$
    \item If {$\CA \geqslant 1$} and $\rho(x) = (1+\alpha|x|)^\beta$ with $\beta \geqslant 1$ and $\alpha>0$, then 
    \begin{align}
        \inf_{T \in \mathbb{R}^+ \cup\{+\infty\}} \kappa_T \leqslant {\CA} K_{f,z} \|\Sigma\| \sqrt{d}\E{\left(\CA+\alpha \CA\left(\frac{\|\Sigma \Sigma^\top\|}{2a}\right)^{1/2} |Y|\right)^{2\beta}}^{1/2} \left(\frac{c_{1,\eqref{eq:constant:growth:sigma-1}}}{a} + \frac{\sqrt{\pi}c_{2,\eqref{eq:constant:growth:sigma-1}}}{\sqrt{a}}\right)
    \end{align}
    where $Y \sim \mathcal{N}(0,I_d)$. Moreover, this upper-bound is an upper-bound for $\kappa_{\infty}$. In particular, $\kappa_{\infty}<1$ as soon as $\alpha$ is small enough and 
    $${\CA} K_{f,z} \|\Sigma\| \sqrt{d} (\CA)^{\beta} \left(\frac{c_{1,\eqref{eq:constant:growth:sigma-1}}}{a} + \frac{\sqrt{\pi}c_{2,\eqref{eq:constant:growth:sigma-1}}}{\sqrt{a}}\right)<1.$$
\end{itemize}

\begin{Remark}
Let us emphasize that the choice of $\rho$ will have an impact on the numerical error of our scheme. In particular, as noticed in Remark \ref{rem:comments-num-error}, we should consider a weight with the largest possible growth in order to minimize the spatial truncation error. Proposition \ref{prop:contraction} says that it is possible to consider an exponential weight when $C_A=1$. On the other hand, if $C_A>1$ we have to settle for a polynomial growth.
\end{Remark}

\end{Proposition}
The proof of Proposition \ref{prop:contraction} is postponed to Section \ref{proof:prop:contraction}. In light of this Proposition, we will consider only the case $T = +\infty$ in the remaining of the paper, which corresponds to equations \eqref{eq:representationZ:infiniteT} and \eqref{eq:representationZ:infiniteT:randomise}.

\begin{Remark}
 When $\Sigma=\sigma I_d$ and $A=a I_d$ with $\sigma>0$ and $a>0$, we can compute that $\CA=1$ and remark that, for all $t>0$ and $h>0$, 
 $$\frac{1}{\sqrt{1-e^{-h}}} \leqslant \frac{1}{\sqrt{1-e^{-t}}} 1_{h>t}+ \sqrt{\frac{t}{1-e^{-t}}}\frac{1_{h \leqslant t}}{\sqrt{h}}.$$
 Then, we obtain the following upper-bound
 \begin{align*}
  \|\Sigma_s^{-1}\|^{1/2} & =  \sqrt{\frac{2a}{\sigma^2 (1-e^{-2as})}} \leqslant \frac{\sqrt{2a}}{\sigma} \left(\frac{1}{\sqrt{1-e^{-t}}}+\sqrt{\frac{t}{1-e^{-t}}}\frac{1}{\sqrt{2as}}\right),\quad \forall t >0. 
 \end{align*}
 Thus, if we consider the weight function $\rho(x) = e^{\alpha |x|}$ with $\alpha>0$, Proposition \ref{prop:contraction} gives us that
\begin{equation}
\label{upperbound-kappa}
    \kappa_{\infty} \leqslant K_{f,z}  \sqrt{\frac{d}{{a}}}\left(2e^{{\frac{\alpha^2 \sigma^2}{a^2}}}F\left(\frac{{\sqrt 2}\alpha \sigma}{{\sqrt a}}\right)\right)^{d/2}  \inf_{t>0} \frac{\sqrt{2}+\sqrt{\pi}\sqrt{t}}{\sqrt{1-e^{-t}}}.
\end{equation}
{A direct numerical computation shows that the above $\inf_{t>0}$ term equals $4.006...<\frac{\sqrt{2}+\sqrt{\pi}}{\sqrt{1-e^{-1}}}=4.008...$; hence, $\kappa_{\infty}$} is smaller than $1$ as soon as $\Kfz\sqrt{\frac{d}{{a}}}{<0.249}$ and $\alpha$ is small enough.
\end{Remark}

\section{Numerical Scheme}
\label{sec:numericalscheme}
The aim of this section is to define a fully implementable scheme in order to provide a numerical approximation of the function $v$ solution to the fixed point equation $v=\Phi_{\infty}(v)$. The proposed scheme is provided in Definition \ref{def:scheme}: it relies on the contraction property given by Proposition \ref{pr:Lpestimate}, a space discretization through a regular grid and an empirical mean appoximation of the expectation. A full study of the approximation error is obtained in Proposition \ref{prop:numericalerror} and Corollary \ref{prop:errornum:ergodic}. 


\subsection{Definition of the scheme}
We denote $\Pi$ a non empty finite subgrid of $\delta \mathbb{Z}^d$, $N$ its cardinality, $\delta$ its mesh size and $\Box$ its convex hull (in $\bR^d$). Without loss of generality we can assume that $0 \in \Box$. 
In order to define a multilinear interpolation procedure on $\Pi$, we consider the following basis functions:
$$\forall z \in \Pi, \quad\forall x \in \bR^d,\quad \psi_z(x):=\prod_{i=1}^d \left(1-|\delta^{-1}(x_i-z_i)|\right)_+,$$
where $(.)_+$ denotes the positive part function.
For a function $\phi : \Pi \rightarrow \bR^{1 \times d}$, we define $P\phi$ the interpolation of $\phi$ on $\bR^d$ as follows
\begin{equation*}
	P\phi(x)=:
	\begin{cases}
		\sum_{z \in \Pi} \psi_z(x)\phi(z), & x \in \Box \\
		\sum_{z \in \Pi} \psi_z(\text{Proj}(x,\Box))\phi(z), & x\notin \Box\,. 
	\end{cases}
\end{equation*}
By a small abuse of notation we also consider this interpolation operator for functions $\phi :\bR^d \rightarrow \bR^{1 \times d}$ by defining $P\phi := P\phi_{|\Pi}$.
By the same definition, $P$ can act also on vector-valued and matrix-valued functions. 

This interpolation operator satisfies following standard properties.
\begin{Proposition}
\label{prop:properties-P}
Let  us consider $\phi, \tilde{\phi} : \bR^d \rightarrow \bR^{1\times d}$. Then we have
\begin{enumerate}
    \item $$\|P\phi - P\tilde{\phi}\|_{\infty} \leqslant \|\phi -\tilde{\phi}\|_{\infty}\,.$$
    \item If $\phi\in C^2(\bR^d,\bR^{1 \times d})$, then there exists $C\geqslant 0$ only depending on $|\nabla^2 \phi|_{\infty, \Box}$ such that
    $$\|P\phi-\phi\|_{\infty,\Box} \leqslant C\delta^2.$$
    \item
    \begin{equation}
        \label{prop:P:upperbound} 
        |P\phi(x)| \leqslant P\rho (x)\sup_{z \in \Pi}\left|\frac{\phi(z)}{\rho (z)}\right|,\quad\forall x \in \bR^d.
    \end{equation}
    \item
    \begin{equation}
        \label{prop:Prho/rho} 
        \sup_{x \in \bR^d} \left| \frac{P \rho(x)}{\rho(x)}\right| \leqslant \sup_{x \in \Box, y \in \Box, |y-x| \leqslant \sqrt{d} \delta} \frac{\rho(y)}{\rho(x)}.
    \end{equation}
\end{enumerate}
\end{Proposition}

    \begin{proof}
        The two first points are standard. For the third point, we have, for all $x \in \bR^d$,
     \begin{eqnarray*}
      |P\phi(x)| &=&\left|\sum_{z \in \Pi} \psi_z(\text{Proj}(x,\Box))\phi(z)\right| \leqslant \sum_{z \in \Pi} \psi_z (\text{Proj}(x,\Box))\rho(z) \left| \frac{\phi(z)}{\rho(z)}\right|\\
      &\leqslant& \sup_{z \in \Pi} \left|\frac{\phi(z)}{\rho(z)}\right| \sum_{z' \in \Pi} \psi_{z'} (\text{Proj}(x,\Box))\rho(z') =  P\rho (x) \sup_{z \in \Pi}\left|\frac{\phi(z)}{\rho (z)}\right|.
     \end{eqnarray*}
     So, it just remains to prove the fourth point.  
     Since $0 \in \Box$ and $\rho$ is an increasing function with respect to $|x|$, then $\frac{P \rho(x)}{\rho (x)}\leqslant 1$ for all $x \notin \Box$. When $x \in \Box$,
     $$|P \rho(x)| \leqslant \sup_{z \in \Pi, \sup_{1\leqslant j \leqslant  d} |x_j-z_j|<\delta} \rho(z) \leqslant \sup_{y \in \Box, |y-x|<\sqrt{d} \delta} \rho(y)$$
     which proves the result.\qed
    \end{proof}


We are now able to define our scheme in the next definition. The idea is to use the Picard iteration, to use a spatial approximation of functions onto $\Pi$ and to approximate expectations by empirical means. Let us remark that the sample size of empirical means will depend on the point $z \in \Pi$ where we compute an approximation of our solution $v(z)$: we denote it $M_z$. 
Finally, we also denote $\lfloor .\rfloor_{R}$ the projection onto the Euclidean ball $\bar{B}(0,R)$ of $\bR^{1 \times d}$.

\begin{Definition}
    \label{def:scheme}
     We construct a sequence of random functions $v^{n}_M : \Omega \times \Pi \rightarrow \bR^{1\times d}$, $n \in \mathbb{N}$ such that $v^0_M=0$ and, for all $n \in \mathbb{N}$, $z\in \Pi$,
\begin{equation}
 \label{def:defscheme1}
 v^{n+1}_M(z)=\left\lfloor\frac1{M_{z}} \sum_{j=1}^{M_{z}} R^{z}_{n+1,j}(Pv^n_M)\right\rfloor_{B} , 
\end{equation}
where $B \geqslant \|v\|_{\infty}$, for any $\phi: \bR^d \rightarrow \bR^{1\times d}$,  $(R_{n,j}^{z}(\phi))_{n,j \in \mathbb{N}^*,z \in \Pi}$ are independent random variables
and for any $z \in \Pi$, $(R_{n,j}^{z}(\phi))_{n,j \in \mathbb{N}^*}$ have the same distribution as 
\begin{equation}
    \label{def:Rz}
R^{z}(\phi):=\frac{\sqrt{\pi}}{\theta}\sqrt{G}e^{-(\frac a\theta-1)G}\tilde{U}_{\frac{G}\theta} f\left(X^z_{\frac{G}\theta},{ \phi(X^z_{\frac{G}\theta}) \Sigma}\right),
\end{equation}
recalling that $\theta \in (0,a)$ and $G\overset d=\cG(1/2,1)$ is independent of $W$.
\end{Definition}

\begin{Remark}
The condition $B \geqslant \|v\|_{\infty}$ gives a weak dependence of the scheme on the unknown solution $v$. Nevertheless, it is possible to use some theoretical a priori estimates (see e.g. the proof of Proposition \ref{prop:existence:uniqueness:EBSDE}) to set $B$. However, it seems that the truncation procedure is not necessary in practice since we observe the convergence in our numerical experiments without applying the truncation. 
\end{Remark}
\begin{Remark}
{The scheme given by Definition \ref{def:defscheme1} is fully implementable and we are able to give a complete study of the numerical error in Section \ref{subsec:theoreticalstudyscheme}: see Corollary \ref{prop:errornum:ergodic}. On the other hand, it is well known that a grid spacial approximation has a major drawback: The size of the grid exponentially increases when dimension $d$ linearly increases and thus, it is not possible in practice to get a numerical scheme that works as soon as $d$ is too large. Obviously, it is also possible to replace the grid spacial approximation by an other dimensional robust spacial approximation as a neural network for example, even if, in this case, the theoretical study of the numerical error would be more complicated. The numerical study of a scheme based on a neural network spacial approximation is left for future works.}
\end{Remark}
\begin{Remark}
Actually, the independence in $z$ could be removed since it does not affect the subsequent error analysis. Having common randomness for different $z$ is refereed to "Common Random Number" method in the literature \cite{glas:yao:92} and may give especially good results in sensitivity analysis problems. In our case, it would contribute to get a smoother (in space) numerical solution.
\end{Remark}

\subsection{Theoretical study of the scheme}
\label{subsec:theoreticalstudyscheme}
In order to treat the statistical error coming from the replacement of expectations by empirical means, we will consider {measure-concentration inequalities based on} Orlicz norms. We denote $\widetilde{\Psi} : \bR^+ \rightarrow \bR^+$ an Orlicz function, that is a continuous non-decreasing function, vanishing in zero and with $\lim_{x \rightarrow + \infty} \widetilde{\Psi}(x)=+\infty$ and we define the $\widetilde{\Psi}$-Orlicz norm of a real random vector $Y$ by 
$$|Y|_{\widetilde{\Psi}} := \inf \left\{ c >0, \mathbb{E}\left[\widetilde{\Psi}\left(\frac{|Y|}{c}\right)\right] \leqslant 1 \right\}.$$
We easily generalize previous defintion to matrix-valued random variables. We also assume that $\widetilde{\Psi}$ is convex\footnote{convex Orlicz functions are also referred to ``Young functions'' or ``N-functions'' in the literature.}, which implies in particular that $| . |_{\widetilde{\Psi}}$ is a norm, and an increasing function in order to insure that $\widetilde{\Psi}^{-1}$ is a concave function defined on $\bR^+$. 

In the remaining, we will use the following convex and increasing Orlicz function: 
\begin{align}
\Psi(.) := \exp(.) -1.
\label{eq:Psi:Orlicz:exp}
\end{align}
Let us remark that $|Y|_{\Psi}<+\infty$ implies that there exists $\varepsilon>0$ such that $\E{e^{\varepsilon |Y|}}<+\infty$ which means that $Y$ is light-tailed.
This Orlicz function satisfies some important properties that are recalled in the next proposition. 
\begin{Proposition}
\label{prop:Talagrand-maximal}
\begin{enumerate}
 \item {\bf [Talagrand inequality]} There exists a universal constant $C_{\Psi}$ such that, for all sequence $(Y_k)_{1 \leqslant k \leqslant K}$ of independent, mean zero, random variables satifying $|Y_k|_{\Psi}<+\infty$ for all $1 \leqslant k \leqslant K$, we have
\begin{equation}
 \label{ineq:Talagrand}
 \left|\sum_{k=1}^K Y_k\right|_{\Psi} \leqslant C_{\Psi} \left( \mathbb{E}\left[\left|\sum_{k=1}^K Y_k\right|\right]+\left|\max_{1\leqslant k \leqslant K} |Y_k| \right|_{\Psi} \right).
\end{equation}
 \item {\bf [Maximal inequality]} There exists a universal constant $C_{\Psi}$ such that, for all sequence $(Y_k)_{1 \leqslant k \leqslant K}$ of random variables satisfying $|Y_k|_{\Psi}<+\infty$ for all $1 \leqslant k \leqslant K$, we have
 \begin{equation}
 \label{ineq:maximal}
  \left|\max_{1\leqslant k \leqslant K} |Y_k| \right|_{\Psi} \leqslant C_{\Psi} \Psi^{-1}(K) \max_{1\leqslant k \leqslant K} |Y_k|_{\Psi}.
 \end{equation}
\end{enumerate}
\end{Proposition}
Talagrand inequality comes from Theorem 3 in \cite{Talagrand-89} while Maximal inequality is provided by Lemma 2.2.2 in \cite{vanderVart-Wellner-96}. We also provide the following technical Lemma, whose proof is postponed to Section \ref{proof:lem:estim:orlicz:ergo}.
\begin{Lemma}
    \label{lem:estim:orlicz:ergo}
     Let us assume that Assumption \ref{eq:generatorfandb'} is fulfilled. 
     There exists $C>0$ that does not depend on $n \in \mathbb{N}$, $M \in \mathbb{N}^*$, $\delta>0$ and $z \in \Pi $ such that, for all $\phi : \mathbb{R}^d \rightarrow \mathbb{R}^{1\times d}$ measurable such that $\|\phi\|_{\infty} \leqslant B$ 
     and $z \in \Pi$,
     \begin{equation}
        \label{ass:finite-Orlicz}
        \left| \frac{R^z(\phi)}{1+|z|} - \E{\frac{R^{z}(\phi)}{1+|z|} } \right|_{\Psi} + \E{\left| \frac{R^z(\phi)}{1+|z|} - \E{\frac{R^{z}(\phi)}{1+|z|} } \right|^2} \leqslant C.
       \end{equation}
    \end{Lemma}

\begin{Proposition}
\label{prop:numericalerror}
Let us assume that Assumption \ref{eq:generatorfandb'} is fulfilled and 
$\kappa_{\infty}<1$, recalling that $\kappa_{\infty}$ is defined in Proposition \ref{pr:Lpestimate}.
In particular, Proposition \ref{pr:Lpestimate} gives us that the fixed point equation $\Phi_{\infty}(v)=v$ has a unique solution $v$.
We also assume that 
$v \in C^2(\bR^d,\bR^{1\times d})$ with bounded second derivatives.


%
%
Then there exists a constant $C>0$ that does not depend on $M$, $n$ and $\delta$ such that 
\begin{align*}
& \E{ \sup_{x \in \bR^d} \left|\frac{ Pv^{n}_M(x) -v(x)}{\rho(x)}\right| }\\
  \leqslant &C \left(\sup_{x \in \bR^d} \left|\frac{P\rho(x)}{\rho(x)} \right|\inf_{z \in \Pi} \frac{\log({1+}N)(1+|z|)}{\sqrt{M_z} \rho(z)}+\delta^2
 +  \frac{1}{\inf_{x \in \partial \Box}  \rho(x)}\right)+\kappa_{\infty}^n \E{ \sup_{x \in \bR^d} \left|\frac{v(x)}{\rho(x)}\right| }.
\end{align*}
\end{Proposition}


The proof of Proposition \ref{prop:numericalerror} is postponed to Section \ref{proof:prop:numericalerror}.

\begin{Remark}
\label{rem:comments-num-error}
 The upper bound obtained for the numerical error in Proposition \ref{prop:numericalerror} can be easily analyzed: 
 \begin{enumerate}
 \item The first term is the statistical error coming from the approximation of the expectation by an empirical mean. The growth of $\rho$ allows to decrease the size sample $M_z$ when $|z|$ is large.
 \item The second term is related to the space discretization by a discrete grid $\delta \mathbb{Z}^d$.
 \item  This third term is a truncation error. In order to get a good control on it, we should consider a weight function $\rho$ with large enough growth.
 \item The last term comes from the Picard procedure and it tends to $0$ only if we have a contraction property for $\Phi_{\infty}$, i.e. $\kappa_{\infty}<1$.
 \end{enumerate}
\end{Remark}

We are now able to specify the error given by Proposition \ref{prop:numericalerror} when we assume that our grid $\Pi$ is centered in $0$, and is given by
$$\left\{(i_1\delta,...,i_d \delta) \,|\, i_k \in \{-\widetilde{N},...,\widetilde{N}\}, k \in \{1,...,d\} \right\}$$
for a given $\widetilde{N} \in \mathbb{N}$, which implies that $N=(2\widetilde{N}+1)^d$. We also take $M_z = \widetilde{M}(1+|z|)^2\rho^{-2}(z)$.

\begin{Corollary}
\label{prop:errornum:ergodic}
Let us assume that assumptions of Proposition \ref{prop:numericalerror} are fulfilled.
\begin{itemize}
\item If $C_A=1$ and $\rho(x) = e^{\alpha|x|}$ {($\alpha>0$)}, then we have
  $$\E{ \sup_{x \in \bR^d} \left|\frac{ Pv^{n}_M(x) -v(x)}{\rho(x)}\right| }= O\left( \delta^2+\frac{\log (2+ \widetilde{N})}{\sqrt{\widetilde{M}}} + e^{-\alpha\widetilde{N}\delta}+ \kappa_{\infty}^n \right).$$
\item If $C_A>1$ and $\rho(x) = (1+\alpha |x|)^\beta$ {($\alpha>0, \beta \geqslant 1$)}, then we have
$$\E{ \sup_{x \in \bR^d} \left|\frac{ Pv^{n}_M(x) -v(x)}{\rho(x)}\right| } = O\left( \delta^2+\frac{\log (2+\widetilde{N})}{\sqrt{\widetilde{M}}} + (1+\alpha \widetilde{N} \delta)^{-\beta}+ \kappa_{\infty}^n \right).$$
\end{itemize}
\end{Corollary}

\begin{proof}
    We just have to apply Proposition \ref{prop:numericalerror} and specify some terms in the upper-bound by setting $\rho$.
    
    First of all, if $C_A=1$, then we can take $\rho(x)=e^{\alpha |x|}$ according to Proposition \ref{prop:contraction}. We apply the fourth point of  Proposition \ref{prop:properties-P} to get
    $$\sup_{x \in \bR^d} \left|\frac{P\rho(x)}{\rho(x)} \right|\leqslant \sup_{x \in \Box, y \in \Box, |y-x| \leqslant \sqrt{d} \delta} e^{\alpha |y|-\alpha |x|}\leqslant e^{\alpha \sqrt{d} \delta}.$$
    We also have $\inf_{x \in \partial \Box}  \rho(x) =   e^{\alpha \widetilde{N} \delta}$
    which gives us the result.

    Otherwise, $C_A>1$ and we can take $\rho(x)=(1+\alpha |x|)^\beta$ according to Proposition \ref{prop:contraction}. Then
    $$\sup_{x \in \bR^d} \left|\frac{P\rho(x)}{\rho(x)} \right|\leqslant \sup_{x \in \Box, y \in \Box, |y-x| \leqslant \sqrt{d} \delta} \left( \frac{1+\alpha |y|}{1+\alpha |x|}\right)^\beta \leqslant (1+\alpha\sqrt{d} \delta)^\beta,$$
    and 
    $\inf_{x \in \partial \Box}  \rho(x) =(1+\alpha \widetilde{N} \delta)^{\beta}$ which concludes the proof.
\qed \end{proof}

\section{Numerical experiments}
\label{sec:num:exp}

We use our approach to solve numerically the ergodic BSDE:
\begin{equation*}
 Y_t = Y_T + \int_t^T \left(f(X_s,Z_s)-\lambda\right) \ds -\int_t^T Z_s \dW_s,\quad 0\leqslant t \leqslant T,
\end{equation*}
where $X$ is an Ornstein-Uhlenbeck process in dimension $d$,  $\sigma=I_d$ and $A=a I_d$, 
$$f(x,z) =1+ \sin(\gamma(|x| + |z|)) +\gamma |z| - \sin(\gamma(|x|+2 |x| e^{-|x|^2})) - (2\gamma|x|+2|x|^2- d+2 a |x|^2)e^{-|x|^2}. $$

We can easily check that the unique solution is given by $Y_t = u(X_t)$, $Z_t = v(X_t)$, $\lambda=1$, where
$u(x) = e^{- |x|^2}$ and $v(x)=-2 x^{{\top}} e^{- |x|^2}$.
Moreover, $X^x_t \sim \mathcal{N}(e^{-at}x,\frac{1-e^{-2at}}{2a}{I_d})$.
In all our numerical experiments we consider a number of Monte-Carlo samples $M_z$ that does not depend on $z$, denoted $M$ in the following. Moreover, we do not apply the truncation step, i.e. we consider the case $B=+\infty$.

\paragraph{\it Dimension 1.} Figure \ref{fig:dim1} illustrates the convergence of our algorithm in dimension $1$. We remark that the Picard scheme has almost converged at the third iteration. Moreover, the truncation of the domain has an impact only on the two extreme points of the grid. {Thus, for this reason, in all numerical experiments, errors are measured on the grid points except on the $r$ extreme ones for various values of $r$, i.e. by setting}
\begin{align*}
\mathfrak{E}^{d,r}_{\infty,n}&:=\sup  \mathcal{E}^{d,r}_{n}, \\ 
\mathcal{E}^{d,r}_{n} &:= \left\{|v^n_{M}(x)-v(x)| : x=(i_1\delta,...,i_d \delta), i_k \in \{-(\widetilde{N}-r),...,(\widetilde{N}-r)\}, k \in \{1,...,d\}  \right\},
\end{align*}
{and for the next experiments in dimension $1$ we consider the sup error $\mathfrak{E}^{1,1}_{\infty,n}$.}

\begin{figure}
    \centering
    \includegraphics[width=14cm]{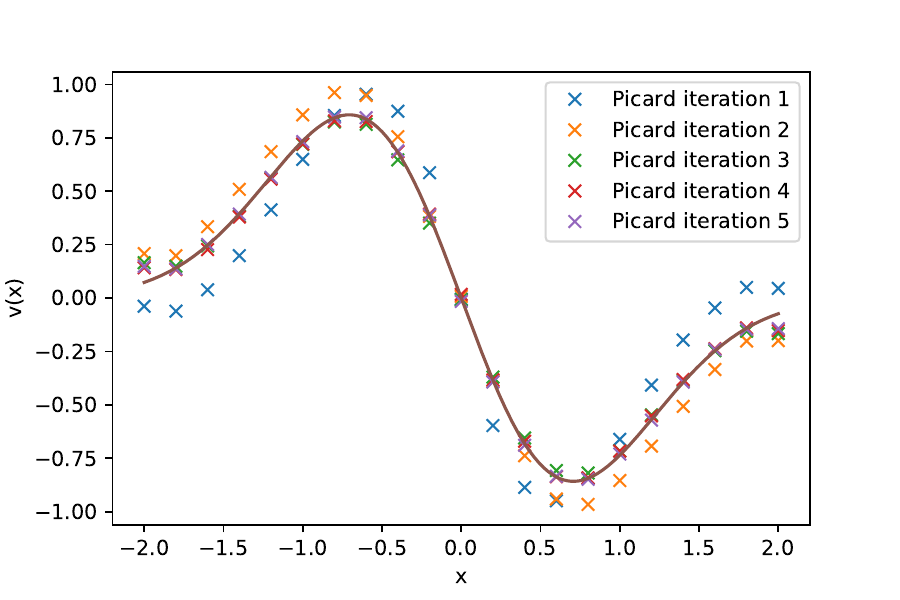}
    \caption{{Solution $v$ at different iterations. Parameters:}  $d=1$, $\gamma=1$, $a=2$, $\theta=1.8$,  $\widetilde{N}=10$, $\delta=0.2$, $M=10^5$.}
    \label{fig:dim1}
\end{figure}

\begin{figure}
    \centering
    \includegraphics[width=14cm]{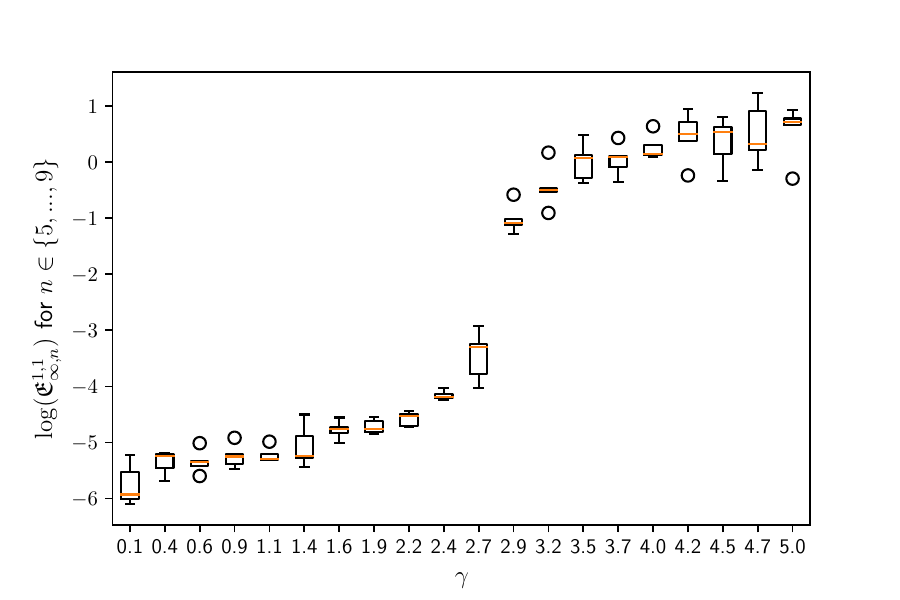}
    \caption{{Box plots of log-sup errors $\mathfrak{E}^{d,r}_{\infty,n}$ (with $d=1$, $r=1$) for different $n$, as a function of $\gamma$. Parameters:} 
    $a=2$, $\theta=1.8$,  $\widetilde{N}=10$, $\delta=0.2$, $M=10^5$.}
    \label{fig:gamma}
\end{figure}

Figure \ref{fig:gamma} shows the impact of $\gamma$ on the numerical convergence of our scheme. We can easily check that $K_{f,z}=2\gamma$. Unsurprisingly, $\gamma$ has an impact on the contraction of the Picard iteration and we can clearly observe the change of behaviour around $\gamma=2.7$. Note that the theoretical upper-bound \eqref{upperbound-kappa} for $\kappa_{\infty}$ is approximately equal to $ 4\sqrt{2}\gamma$ when $\alpha$ tends to $0$ which implies the contraction when $\gamma <(4\sqrt{2})^{-1}$: this theoretical bound is clearly far behind the threshold numerically observed.

Figure \ref{fig:theta} illustrates the impact of the choice of $\theta$ on the numerical convergence of our scheme. {Intuitively, $\theta$ plays a role on the statistical fluctuation of our scheme.} From our theoretical study, we have seen that we should take $\theta<a$. In practice, the threshold is softer: taking $\theta$ too large is clearly a bad idea but there is no numerical problems up to $\theta=3$ when $a=2$.

\begin{figure}
    \centering
    \includegraphics[width=14cm]{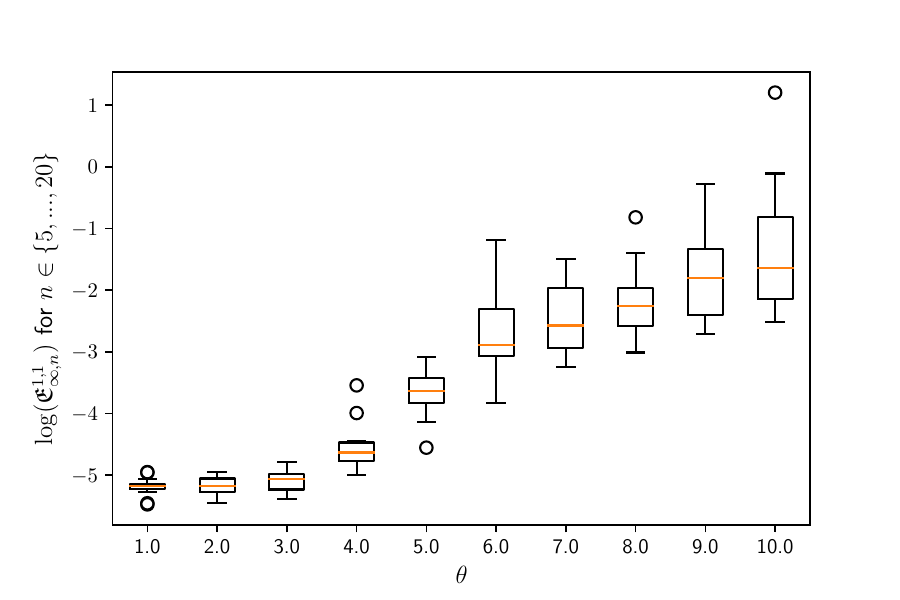}
    \caption{{Box plots of log-sup errors $\mathfrak{E}^{d,r}_{\infty,n}$ (with $d=1$, $r=1$) for different $n$, as a function of $\theta$. Parameters:} $a=2$, $\gamma=1$,  $\widetilde{N}=10$, $\delta=0.2$, $M=10^5$.}
    \label{fig:theta}
\end{figure}

\paragraph{\it Dimension 2.} The impact of $\gamma$ in dimension $2$ is illustrated in Figure \ref{fig:gamma-d2}. We observe that the threshold $\gamma \approx 2.7$ obtained in dimension $1$ remains of the same order in dimension $2$. Let us also remark that when we are close to the threshold, the error due to the truncation of the domain seems to propagate further since the sup error $\mathfrak{E}^{2,2}_{\infty,n}$ becomes smaller than $\mathfrak{E}^{2,1}_{\infty,n}$.

\begin{figure}
    \centering
    \includegraphics[width=7cm]{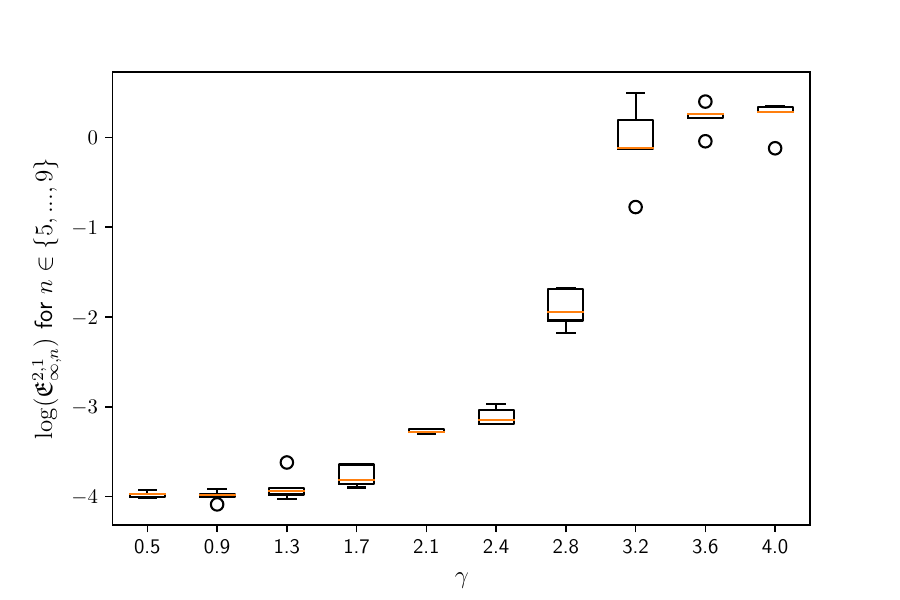}
    \includegraphics[width=7cm]{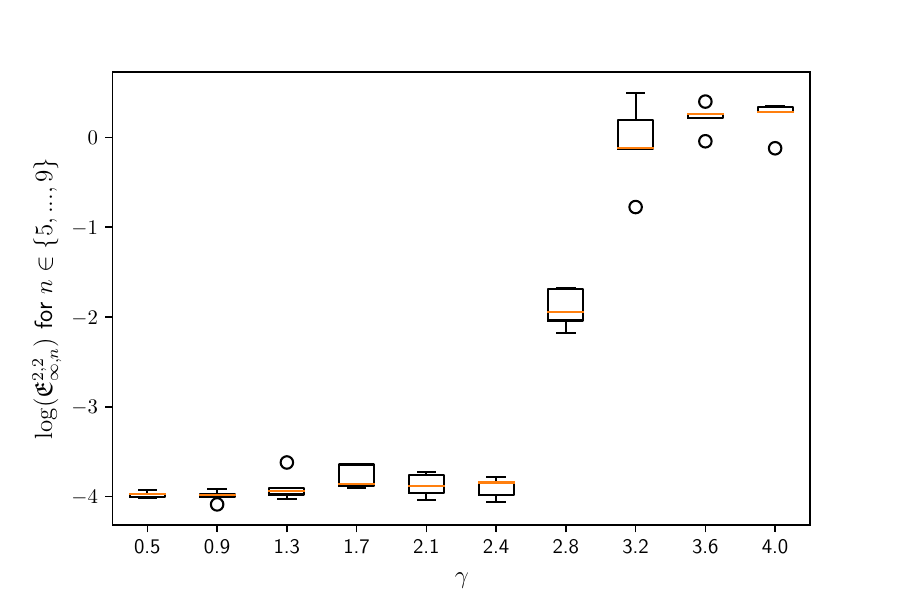}
    \caption{{Box plots of log-sup error $\mathfrak{E}^{d,r}_{\infty,n}$ (with $d=2$) for different $n$, as a function of $\gamma$. On the left $r=1$, on the right: $r=2$. Other parameters:} $a=2$, $\theta=2$, $\widetilde{N}=10$, $\delta=0.2$, $M=10^5$.}
    \label{fig:gamma-d2}
\end{figure}
\FloatBarrier

\paragraph{\it Higher dimension.} Finally, we also investigate the impact of the dimension in Figure \ref{fig:multi-dim} and Table \ref{tab:time}. Note  that we have reduced $M$ and $\widetilde{N}$ and we have increased $\delta$ with respect to previous numerical experiments. As previously remarked, $\mathfrak{E}^{d,1}_{\infty,3}$ is much better than $\mathfrak{E}^{d,0}_{\infty,3}$ whereas $\mathfrak{E}^{d,2}_{\infty,3}$ is of the same order as $\mathfrak{E}^{d,1}_{\infty,3}$. Numerical experiments have been conducted on a Intel(R) Core(TM) i5-4590S CPU @ 3.00GHz with 8GB RAM, by using Python 3 with Numba library but without parallelization. Obviously, due to the curse of dimension, we are not able to tackle dimensions beyond $d=5$ in a reasonable time. Nevertheless, it should be possible to implement our scheme by using parallelization paradigm, on CPU or even GPU, and then to increase a little bit the upper bound on covered dimensions. Moreover, it should be also possible to use neural networks instead of grid in order to tackle high dimensional problems. These two research directions are left for future works.

\begin{figure}
    \centering
    \includegraphics[width=13cm]{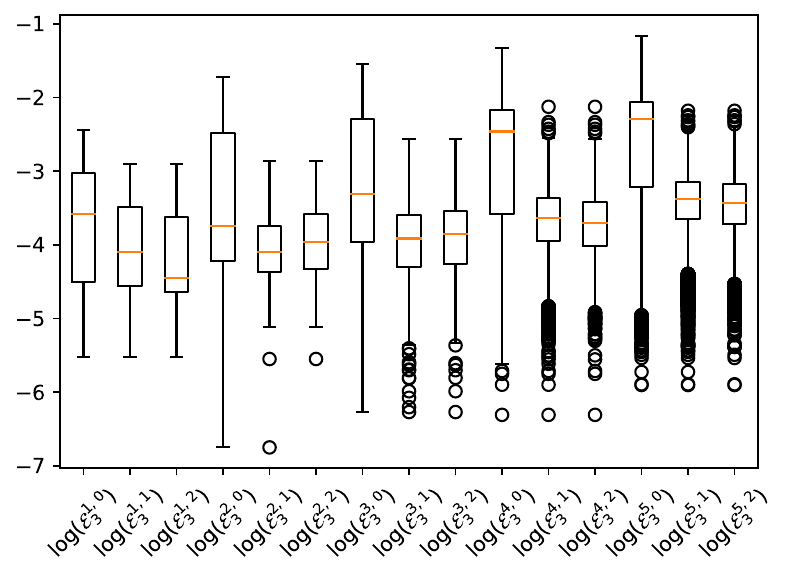}
    \caption{{Box plots of log grid errors $\mathcal{E}^{d,r}_{n}$ (with $d=1,2, 3, 4, 5$, $r=0, 1, 2$). Parameters:} 
    $a=2$, $\gamma=1$ $\theta=1.8$, $\widetilde{N}=5$, $\delta=0.4$, $M=10^4$}
    \label{fig:multi-dim}
\end{figure}

\begin{table}
    \centering
    \begin{tabular}{|c||c|c|c|c|c|}
        \hline
        Dimension $d$ & 1& 2 & 3 & 4 &5  \\
        \hline 
        Error $\mathfrak{E}^{d,1}_{\infty,3}$ & $5.49* 10^{-2}$ & $5.69* 10^{-2}$ & $7.69* 10^{-2}$ & $11.9* 10^{-2}$ & $11.3* 10^{-2}$\\
        \hline
        Error $\text{Mean}(\mathcal{E}^{d,1}_{3})$ & $2.31* 10^{-2}$ & $1.94* 10^{-2}$ & $2.16* 10^{-2}$ & $2.76* 10^{-2}$ & $3.49* 10^{-2}$\\
        \hline
        Time ($s$) & $4$ & $18$ & $217$ & $4155$ & $86639$ \\
        \hline
    \end{tabular}
    \caption{{Comparison of sup errors and computational times as a function of the dimension $d$. Parameters:} $a=2$, $\gamma=1$ $\theta=1.8$, $\widetilde{N}=5$, $\delta=0.4$, $M=10^4$}
    \label{tab:time}
\end{table}

\clearpage



\section{Proofs}

{\subsection{Proof of Theorem \ref{thm:representationBSDE:ergodic}}

\label{proof:thm:representationBSDE:ergodic}
\paragraph{\it Proof of item \eqref{thm:representationBSDE:ergodic:1}.}
Start from \eqref{eq:EBSDE:markov}, apply Fubini theorem and rewrite expectation using the transition probability density of $X$: it gives, for all $x \in \bR^d$ and $T>0$,
\begin{align}\label{eq:u:1}
 u(x)= \int_{\bR^d} u(y) p(0,x;T,y) \dy + \int_0^T \left(\int_{\bR^d} f(y,\bar{u}(y)) p(0,x;s,y) \dy -\lambda\right) \ds.
\end{align}
The density $p(0,x;s,y)$ is smooth in $x,y,s$ provided that $s>0$. 
Denote by $g(c,y)$ the density at point $y\in \bR^d$ of the Gaussian distribution $\cN(0,c\Id)$. Then, leveraging the bounds  \eqref{eq:bound:Sigma}, a direct and standard computation shows that, for any  compact set $K \subset \bR^d$, there exist positive constants $c_1,c_2$ such that, for all $s \in (0,T]$, $y \in \bR^d$, we have}
\begin{align*}
\sup_{x \in K} |\nabla_x p(0,x;s,y)| &\leq  c_1 \min (1,s^{-1/2}) g(s c_2,y).
\end{align*}
{In addition, owing to the bounds \eqref{eq:prop:existence:uniqueness:EBSDE:bounds} on $u$ and $\bar u$, the functions 
\begin{align}
    y\mapsto & \min (1,T^{-1/2}) g(T c_2,y) \norm{u(y)}\\
    \text{and}\qquad
(s,y)\mapsto & \min (1,s^{-1/2}) g(s c_2,y)\norm{f(y,\bar{u}(y))}
\end{align}
are clearly integrable on $\bR^d$ and $[0,T]\times\bR^d$ respectively. Thus, $u$ as given in \eqref{eq:u:1} is $C^1$. By rewriting the first term on the right hand side of \eqref{eq:u:1} as an expectation, differentiating with respect to $x$, using $\nabla_x X_T^x=e^{-AT}$ (see \eqref{eq:OU:bis}) and the above arguments for the time integral, it readily follows
\begin{align}\label{eq:u:2}
 \nabla_x u(x)=\E{\nabla_x u(X^x_T)e^{-AT}} + \int_0^T \int_{\bR^d} f(y,\bar{u}(y)) \nabla_x  p(0,x;s,y) \dy  \ds.
\end{align}}

{\paragraph{\it Proof of item \eqref{thm:representationBSDE:ergodic:2}.}
Notice that for fixed $T$, one can look at \eqref{eq:EBSDE:intro} as a classical BSDE in a  finite horizon $[0,T]$, with terminal condition $u(X_T)$. Thus, the standard results \cite[Theorem 3.1]{ma:zhan:02:1} apply and give the first announced result. The boundedness of $v$ comes from Proposition \ref{prop:existence:uniqueness:EBSDE} and the invertibility of $\Sigma$.}

{\paragraph{\it Proof of \eqref{eq:representationZ:finiteT}.} 
We have $ \nabla_x \log(p(0,x;T,y))=\frac{\nabla_x p(0,x;T,y)}{p(0,x;T,y) }=(y-e^{-A s} x )^\top\Sigma_s^{-1} e^{-A s}$: therefore, passing from \eqref{eq:u:2} to \eqref{eq:representationZ:finiteT} is granted by Fubini theorem. {Indeed, starting from \eqref{eq:barU:tildeU} and using the Lipschitz property of $f$ with \eqref{eq:bound:esp:Utilde}-\eqref{eq:bound:Sigma}-\eqref{eq:bound:acc:X},
we get}
\begin{align}
\E{\norm{f(X_s^x, v(X_s^x)\Sigma)\bar{U}^x_s}} 
&{\leqslant C \E{(1+\norm{X_s^x})(1 \lor s^{-1/2})}e^{-as}} \\
&\leqslant  C(1 \lor s^{-1/2})  e^{-as}
\label{eq:exp:decay}
\end{align}
with a new constant (depending on $x$ but uniform in $s$) at each line.
Using $\bar U^x_s=e^{-as} \tilde{U}_s$ {(see \eqref{eq:barU:tildeU})}, we complete the proof of \eqref{eq:representationZ:finiteT}.}

\paragraph{\it Proof of \eqref{eq:representationZ:finiteT:randomise}.} 
This is obtained by reinterpreting the time-integral \eqref{eq:representationZ:finiteT} as an integral with respect to the distribution of $E\overset d=\cG(\frac{1}{ 2},\theta)$: for any integrable function $\varphi:[0,T]\mapsto \bR$, 
\begin{align}
\int_0^T \varphi(s) e^{-as}\ds&=\E{\1_{E\leq T} \varphi(E) e^{-aE+\theta E} \frac{\sqrt \pi}{\sqrt \theta} \sqrt E}\\
&=\E{\1_{G\leq T \theta } \varphi\left(\frac{G}{\theta }\right) e^{-(a-\theta)/\theta G} \frac{\sqrt \pi}{\theta} \sqrt G}
\end{align}
using $G:=E\ \theta\overset d=\cG(\frac{1}{ 2},1)$. Then \eqref{eq:representationZ:finiteT:randomise} easily follows with $\varphi(s)=\tilde{U}_s f(\Xx_s,v(\Xx_s) \Sigma)$.

\paragraph{\it Proof of \eqref{eq:representationZ:infiniteT} and \eqref{eq:representationZ:infiniteT:randomise}.} 
Their proofs follow by passing to the limit in \eqref{eq:representationZ:finiteT} and \eqref{eq:representationZ:finiteT:randomise} as $T\to+\infty$. This  is possible thanks to the exponential decay in \eqref{eq:exp:decay} and to \eqref{eq:bound:expA}. 
\qed

\subsection{Proof of Proposition \ref{pr:Lpestimate}}
\label{proof:pr:Lpestimate}
Let us start by proving that $\Phi_T(w) \in C^0(\mathbb{R}^d,\mathbb{R}^{1 \times d})$ when we assume that $w \in C^0_{\rho}(\mathbb{R}^d,\mathbb{R}^{1 \times d})$. For any compact set $K \subset \mathbb{R}^d$ we have, using \eqref{eq:growth:rho} and the fact that $X_T^x-e^{-AT}x$ is a Gaussian random variable that does not depend on $x$,
\begin{align*}
    \E{\sup_{x \in K} |w(X_T^x)|} \leqslant C \E{\sup_{x \in K} e^{C|X_T^x - e^{-AT}x|}e^{C|e^{-AT}x|}} \leqslant C_K \E{e^{C|X_T^0|}}<+\infty.
\end{align*}
By using the linear growth of $f$, \eqref{eq:growth:rho} and same computations as previously, recalling \eqref{eq:bound:esp:Utilde} and the fact that $\tilde{U}$ does not depend on $x$, we also get
\begin{align*}
    &\int_0^T \E{\sup_{x \in K} \big|e^{-as}\tilde{U}_s f(X_s^x, w(X_s^x)\Sigma){\big|}}\ds\\ \leqslant&   C \int_0^T e^{-as} \E{\sup_{x \in K} e^{C|X_s^x - e^{-As}x|}e^{C|e^{-As}x|}}^{1/2} \E{|\tilde{U}_s|^2}^{1/2}\ds\\
    \leqslant& C_K \int_0^T e^{-as} (1 \wedge s^{-1/2})\ds<+\infty.
\end{align*}
Then, Lebesgue's dominated convergence theorem gives us that $\Phi_T(w)$ is a continuous function. Same kind of computations lead us also to 
\begin{align*}
    \frac{|\Phi_T(w)(x)|}{\rho(x)} \leqslant& \|e^{-AT}\|\E{\frac{|w(X_T^x)|}{\rho(X_T^x)}\frac{\rho(X_T^x)}{\rho(x)}}\1_{T<+\infty}\\
    &+ \int_0^T e^{-as}C\E{|\tilde{U}_s|\frac{1+|X_s^x|+|w(X_s^x)|}{\rho(X_s^x)}\frac{\rho(X_s^x)}{\rho(x)}}\ds\\
    \leqslant& \CA e^{-aT} c_{T,\eqref{eq:condition:normA:rho:1}}\normr{w} +{ \CA} C(1+\normr{w})c_{T,\eqref{eq:condition:normA:rho:2}}<+\infty
\end{align*}
which implies that $\Phi_T(w) \in C^0_{\rho}(\mathbb{R}^d,\mathbb{R}^{1\times d})$. 
Now, we consider $w_1,w_2 \in C^0_{\rho}(\mathbb{R}^d,\mathbb{R}^{1\times d})$. By using the Lipschitz property of $f$ 
and same computations as previously, we have
\begin{align*}
\frac{|\Phi_T(w_1)(x)-\Phi_T(w_2)(x)|}{\rho(x)} \leqslant& \E{\|e^{-AT}\| \normr{w_1-w_2}\frac{\rho(X_T^x)}{\rho(x)}}\1_{T <+\infty}\\
& + K_{f,z}\int_0^T e^{-as}\E{|\tilde{U}_s|\normr{w_1-w_2}\|\Sigma\|\frac{\rho(X_s^x)}{\rho(x)}}\ds\\
\leqslant & \CA e^{-aT}c_{T,\eqref{eq:condition:normA:rho:1}}\normr{w_1-w_2}\1_{T <+\infty}\\
& + K_{f,z} \|\Sigma\| \CA\int_0^T e^{-as}\E{|\Sigma_s^{-1}(X_s^x-e^{-As}x)|\frac{\rho(X_s^x)}{\rho(x)}}\ds \normr{w_1-w_2}\\
\leqslant & \left({\CA} e^{-aT}c_{T,\eqref{eq:condition:normA:rho:1}}\1_{T <+\infty} + K_{f,z} \|\Sigma\| \CA c_{T,\eqref{eq:condition:normA:rho:2}}\right) \normr{w_1-w_2}.
\end{align*}
The end of the Proposition is a straightforward application of the Banach fixed-point theorem.
\qed

\subsection{Proof of Proposition \ref{prop:contraction} }
\label{proof:prop:contraction}

We start by assuming that $\CA=1$ and $\rho(x) = e^{\alpha |x|}$ with $\alpha>0$. Then, by considering $Y \sim \mathcal{N}(0,I_d)$, we have 
\begin{align}
    c_{T,\eqref{eq:condition:normA:rho:1}} \leqslant & \sup_{x \in \mathbb{R}^d} \E{e^{\alpha |X_T^x-e^{-AT}x|} e^{\alpha(|e^{-AT}x|-|x|)}} \leqslant \E{e^{\alpha |\Sigma_T^{1/2} Y|}}
    \leqslant \E{e^{\alpha \|\Sigma_T\|^{1/2} |Y|}} \\
    \leqslant & \E{e^{\frac{\alpha \|\Sigma \Sigma^\top\|^{1/2}}{\sqrt{2a}} \sum_{i=1}^d |Y_i|}} = \left(2e^{\frac{\alpha^2 \|\Sigma \Sigma^\top\|}{4a}}F\left(\frac{\alpha \|\Sigma \Sigma^\top\|^{1/2}}{\sqrt{2a}}\right)\right)^d
    \label{eq:bound:cT1}
\end{align}
where we have used that $C_A=1$ in the first inequality, {$\E{e^{\lambda |Y_1|}}=2 e^{\lambda^2/2}F(\lambda)$} and $F$ stands for the cumulative distribution function of the Gaussian distribution $\mathcal{N}(0,1)$. 
By the same token, we also get
\begin{align}
    c_{T,\eqref{eq:condition:normA:rho:2}} \leqslant & \int_0^T e^{-as}\sup_{x \in \mathbb{R}^d} \E{|\Sigma_s^{-1}(X_s^x - e^{-As}x)|e^{\alpha|X_s^x - e^{-As}x|}e^{\alpha (|e^{-As} x|-|x|)}}\ds\\
    \leqslant & \int_0^T e^{-as} \E{|\Sigma_s^{-1/2}Y|^2}^{1/2}\E{e^{2\alpha |\Sigma_s^{1/2} Y|}}^{1/2}\ds\\
    \leqslant & \int_0^T e^{-as} \|\Sigma_s^{-1}\|^{1/2}\sqrt{d}\left(2e^{\frac{\alpha^2 \|\Sigma \Sigma^\top\|}{a}}F\left(\frac{\sqrt{2}\alpha \|\Sigma \Sigma^\top\|^{1/2}}{{\sqrt a}}\right)\right)^{d/2}\ds.
    \label{eq:bound:cT2}
\end{align}
Then, we can use \eqref{eq:bound:cT1}, \eqref{eq:bound:cT2} and \eqref{eq:constant:growth:sigma-1} to get the upper bound
\begin{align*}
\kappa_T \leqslant & \left(2e^{\frac{\alpha^2 \|\Sigma \Sigma^\top\|}{4a}}F\left(\frac{\alpha \|\Sigma \Sigma^\top\|^{1/2}}{\sqrt{2a}}\right)\right)^d e^{-aT}\\
&+ K_{f,z} \|\Sigma\| \sqrt{d}\left(2e^{\frac{\alpha^2 \|\Sigma \Sigma^\top\|}{a}}F\left(\frac{\sqrt{2}\alpha \|\Sigma \Sigma^\top\|^{1/2}}{{\sqrt a}}\right)\right)^{d/2} \int_0^T e^{-as} \left(c_{1,\eqref{eq:constant:growth:sigma-1}} + \frac{c_{2,\eqref{eq:constant:growth:sigma-1}}}{\sqrt{s}}\right) \ds.
\end{align*}
A simple study of this upper-bound as a function of $T$ shows that this upper-bound is minimal in $0$ or in $+\infty$.
Moreover, this upper-bound is bigger than $1$ for $T=0$ which never gives us a contraction. Computing the integral when $T=+\infty$ gives us the bound written in the proposition. 

We assume now that {$\CA\geqslant 1$} and $\rho(x)=(1+\alpha|x|)^{\beta}$ with $\beta\geqslant 1$ and $\alpha>0$. We have now
\begin{align}
    c_{T,\eqref{eq:condition:normA:rho:1}} \leqslant & \sup_{x \in \mathbb{R}^d} \E{\left(\frac{1+\alpha|X_T^x-e^{-AT}x| +\alpha\CA e^{-aT}|x|}{1+\alpha|x|}\right)^\beta} \leqslant \E{\left(\sup_{x \in \mathbb{R}^d}\frac{1+\alpha|X_T^0| +\alpha\CA |x|}{1+\alpha|x|}\right)^\beta} \\
    \leqslant & \E{(\CA+\alpha|\Sigma_T^{1/2} Y|)^{\beta}}
    \leqslant \E{\left(\CA+\alpha\CA\left(\frac{\|\Sigma \Sigma^\top\|}{2a}\right)^{1/2} |Y|\right)^{\beta}}
    \label{eq:bound:cT3}
\end{align}
and, by the same token,
\begin{align}
    c_{T,\eqref{eq:condition:normA:rho:2}} \leqslant & \int_0^T e^{-as}\sup_{x \in \mathbb{R}^d} \E{|\Sigma_s^{-1}X_s^0|\left(\frac{1+\alpha|X_s^0|+\alpha C_A |x|}{1+\alpha|x|}\right)^{\beta}}\ds\\
    \leqslant & \int_0^T e^{-as}\E{|\Sigma_s^{-1/2}Y|^2}^{1/2} \E{\left(\CA+\alpha\CA \left(\frac{\|\Sigma \Sigma^\top\|}{2a}\right)^{1/2} |Y|\right)^{2\beta}}^{1/2}\ds\\
    \leqslant &  \sqrt{d}\E{\left(\CA+\alpha \CA\left(\frac{\|\Sigma \Sigma^\top\|}{2a}\right)^{1/2} |Y|\right)^{2\beta}}^{1/2} \int_0^T e^{-as} \left(c_{1,\eqref{eq:constant:growth:sigma-1}} + \frac{c_{2,\eqref{eq:constant:growth:sigma-1}}}{\sqrt{s}}\right) \ds.
    \label{eq:bound:cT4}
\end{align}
Then, we can do as previously: we use \eqref{eq:bound:cT3} and \eqref{eq:bound:cT4} to get an upper-bound for $\kappa_T$ that we can optimize in $T$. This upper-bound is minimal in $0$ or $+\infty$,
the value in $0$ is bigger than $1$ while the value in $+\infty$ gives the bound advertised in the proposition.
\qed

\subsection{Proof of Lemma \ref{lem:estim:orlicz:ergo}}

\label{proof:lem:estim:orlicz:ergo}

By using the growth of $f$ and $\phi$ as well as Young inequality, \eqref{eq:bound:Sigma} and $\theta<a$, we have for all $c>0$,
\begin{align*} 
\E{ e^{{\frac 1c}\left|\frac{R^z(\phi)}{1+|z|}\right|}}\leqslant & \E{\int_0^{+\infty} e^{{\frac Cc}\sqrt{s}e^{-\left(\frac{a}{\theta}-1\right)s}\|\tilde{U}_s\|\frac{1+|X_s^z|}{1+|z|}} \frac{1}{\sqrt{\pi s}}e^{-s} \ds }\\ 
\leqslant& \E{\int_0^{+\infty} e^{{\frac Cc} \|\Sigma_s^{-1}\| |X_s^z-e^{-A s} z| \left(1+|X_s^z-e^{-A s} z| +\frac{|e^{-As}z|}{1+|z|}\right)}\frac{1}{\sqrt{\pi s}}e^{-s}  \ds }\\
\leqslant &  e^{{\frac Cc}}\int_0^{+\infty} \E{ e^{{\frac Cc} \|\Sigma_s^{-1}\| |X_s^z-e^{-A s} z|^2 }}\frac{1}{\sqrt{\pi s}}e^{-s}  \ds \\
\leqslant &  e^{{\frac Cc}}\int_0^{+\infty} 
\E{ e^{{\frac Cc}  |
{\Sigma_s^{-1/2}}(X_s^z-e^{-A s} z)|^2 }}\frac{1}{\sqrt{\pi s}}e^{-s}  \ds 
\end{align*}
where, as usual, the constant $C$ may change from one term to another but does not depend on $n$, $M$, $\delta$, $z$ {and $c$}. Since 
${\Sigma_s^{-1/2}}(X_s^z-e^{-A s} z) \sim \mathcal{N}(0,I_d)$, 
$\E{ e^{{\frac Cc}  |{\Sigma_s^{-1/2}}(X_s^z-e^{-A s} z)|^2 }} = (1-2{\frac Cc})^{-d/2}$
 as soon as ${\frac Cc}<1/2$. Thus, for all $c>{2C}$,
\begin{align}
\E{ e^{{\frac 1c}\left|\frac{R^z(\phi)}{1+|z|}\right|}}
\leqslant e^{{\frac Cc}}\left(1-2{\frac Cc}\right)^{-d/2}.\label{eq:maj:exp:Rz}
\end{align}
{For some ${\frac Cc}=\tau^\star<1/2$ small enough, the above upper bound is smaller than 2, which proves that \begin{align}\label{eq:norm:Rz}
\left|\frac{R^z(\phi)}{1+|z|}\right|_{\Psi}\leq \frac{C}{\tau^\star}=:c^{\star}.
\end{align}
In addition, combining $x+\frac{x^2}2\leq e^x-1$ for any $x\geqslant 0$ and \eqref{eq:maj:exp:Rz} with $c=c^{\star}$ we get  
\begin{align}\label{eq:moment:Rz}
    \frac 1{c^{\star}}\E{\left|\frac{R^z(\phi)}{1+|z|}\right|}+
    \frac {1}{2(c^{\star})^2}\E{\left|\frac{R^z(\phi)}{1+|z|}\right|^2}
     \leqslant \E{e^{\frac 1{c^{\star}}\left|\frac{R^z(\phi)}{1+|z|}\right|}}-1\leq 1.
\end{align}
This readily yields the second part of \eqref{ass:finite-Orlicz}.

The first part easily follows too, in view of \eqref{eq:norm:Rz}-\eqref{eq:moment:Rz} and since $|.|_{\Psi}$ satisfies the triangular inequality.}
\qed


\subsection{Proof of Proposition \ref{prop:numericalerror}}

\label{proof:prop:numericalerror}

Let us denote, for $n \geqslant 0$,
$$e_{\infty,n+1}:=\E{ \sup_{x \in \bR^d} \left|\frac{ Pv^{n+1}_M(x) -v(x)}{\rho(x)}\right|}.$$
We have
\begin{align*}
 e_{\infty,n+1} \leqslant  \mathcal{E}_{\infty,1} +  \mathcal{E}_{\infty,2} +  \mathcal{E}_{\infty,3}
\end{align*}
with
\begin{align*}
 \mathcal{E}_{\infty,1} &:= \E{ \sup_{x \in \bR^d} \left| \frac{Pv^{n+1}_M(x) -P(\lfloor\mathbb{E}_{v^n_M}[R^{(\cdot)}(Pv^n_M)]\rfloor_{B})(x)}{\rho(x)}\right| },\\
 \mathcal{E}_{\infty,2} &:= \E{ \sup_{x \in \bR^d} \left|\frac{Pv(x)-P(\lfloor\mathbb{E}_{v^n_M}[R^{(\cdot)}(Pv^n_M)]\rfloor_{B})(x)}{\rho(x)}\right|},\\
 \mathcal{E}_{\infty,3} &:= \sup_{x \in \bR^d} \left|\frac{Pv(x)-v(x)}{\rho(x)}\right|.
\end{align*}
Here, the subscript in $\mathbb{E}_{v^n_M}[.]$ means that the expectation is conditionally to all the simulation noises up to iteration $n$.

\paragraph{\it Error $\mathcal{E}_{\infty,1}$:}
Recalling that $P$ is linear, $\lfloor . \rfloor_{\|v\|_{B}}$ is $1$-Lipschitz and applying inequality (3) in Proposition \ref{prop:properties-P}, we get
\begin{align}
 \nonumber \mathcal{E}_{\infty,1} &\leqslant  \E{\sup_{z \in \Pi} \frac{1}{\rho(z)}\left|\left\lfloor \mathbb{E}_{v^n_M}\left[R^{z}(Pv^n_M) \right]\right\rfloor_{B} - \left\lfloor \frac{1}{M_z} \sum_{j=1}^{M_z} R^{z}_{n+1,j}(Pv^n_M)\right\rfloor_{B}\right|}\sup_{x \in \bR^d} \left|\frac{P\rho(x)}{\rho(x)} \right|\\
 &\leqslant \sup_{x \in \bR^d} \left|\frac{P\rho(x)}{\rho(x)} \right| \E{\E{\sup_{z \in \Pi} |H_{z}(\phi)|}_{|\phi = Pv_M^n}}\,,
 \label{ineq:error1}
\end{align}
where 
$$H_z(\phi) := \frac{1}{M_z} \sum_{j=1}^{M_z} \frac{R^{z}_{n+1,j}(\phi)}{\rho(z)} - \E{\frac{R^{z}(\phi)}{\rho(z)} }.$$

Since $\Psi$ is a convex and increasing function, Jensen inequality gives us
\begin{align*}
 \E{\sup_{z \in \Pi} |H_{z}(\phi)|} &\leqslant  \left| \sup_{z \in \Pi} |H_{z}(\phi)| \right|_{\Psi} \Psi^{-1}\left( \mathbb{E}\left[\Psi \left(\frac{\sup_{z \in \Pi} |H_{z}(\phi)|}{\left| \sup_{z \in \Pi} |H_{z}(\phi)|\, \right|_{\Psi}}\right)\right]\right)\\
 & \leqslant  \Psi^{-1}(1) \left| \sup_{z \in \Pi} |H_{z}(\phi)| \right|_{\Psi}.
\end{align*}
Then we upper bound the right hand side of the previous inequality by applying Maximal inequality \eqref{ineq:maximal} to get
\begin{align}
\label{ineq:EsupH}
 \E{\sup_{z \in \Pi} |H_{z}(\phi)|} &\leqslant C \Psi^{-1}(N)\sup_{z \in \Pi} \left| H_{z}(\phi) \right|_{\Psi} {=} C \log ({1+}N) \sup_{z \in \Pi} \left| H_{z}(\phi) \right|_{\Psi}.
\end{align}
Since $H_z(\phi)$ is a sum of centered i.i.d. r.v., we can apply Talagrand inequality \eqref{ineq:Talagrand}:
\begin{align*}
 \left| H_{z}(\phi) \right|_{\Psi} & \leqslant C\left( \mathbb{E}\left[ |H_{z}(\phi)| \right]+ \left| \sup_{1 \leqslant j \leqslant M_z} \frac{\left| \frac{R^z_{n+1,j}(\phi)}{\rho(z)} - \mathbb{E}\left[\frac{R^{z}(\phi)}{\rho(z)} \right] \right|}{M_z} \right|_{\Psi} \right)
\end{align*}
which gives us, using the upper-bound \eqref{ass:finite-Orlicz} and Maximal inequality \eqref{ineq:maximal},
\begin{align*}
 \left| H_{z}(\phi) \right|_{\Psi}  &\leqslant  C \frac{1}{\sqrt{M_z}}\frac{1+|z|}{\rho(z)}\E{ \left| \frac{R^z(\phi)}{1+|z|} - \E{\frac{R^{z}(\phi)}{1+|z|} } \right|^2 }^{1/2}\\
 &\qquad+ C\frac{\Psi^{-1}(M_z)}{M_z}\frac{1+|z|}{\rho(z)}\left| \frac{R^z(\phi)}{1+|z|} - \E{\frac{R^{z}(\phi)}{1+|z|} } \right|_{\Psi}\\
 &\leqslant C M_z^{-1/2}\frac{1+|z|}{\rho(z)}.  
\end{align*}
We just have to plug the previous bound into \eqref{ineq:EsupH} and \eqref{ineq:error1} to get
$$ \mathcal{E}_{\infty,1} \leqslant C\sup_{x \in \bR^d} \left|\frac{P\rho(x)}{\rho(x)} \right| \inf_{z \in \Pi} \frac{\log({1+}N)(1+|z|)}{\sqrt{M_z} \rho(z)}.$$

%
%
%
%
%

\paragraph{\it Error $\mathcal{E}_{\infty,2}$:}

Recalling Proposition \ref{pr:Lpestimate}, we have that $v$ is the unique solution of the fixed point equation $\Phi_{\infty}(v)=v$. We get, by using the linearity of $P$, the first inequality in Proposition \ref{prop:properties-P}, the fact that $\lfloor . \rfloor_{B}$ is $1$-Lipschitz and Proposition \ref{pr:Lpestimate}, 
\begin{equation*}
	\begin{split}
		&\E{ \sup_{x \in \bR^d} \left|\frac{Pv(x)-P(\lfloor\mathbb{E}_{v^n_M}[R^{(\cdot)}(Pv^n_M)]\rfloor_{B})(x)}{\rho(x)}\right|} \\
		 = & \E{ \sup_{x \in \bR^d} \left|\frac{P(\lfloor\Phi(v(\cdot))\rfloor_{B})(x)-P(\lfloor\Phi(Pv^n_M(\cdot))\rfloor_{B})(x)}{\rho(x)}\right|} \\
		 \leqslant  & \E{ \sup_{x \in \bR^d} \left|\frac{\Phi(v(\cdot))(x)-\Phi(Pv^n_M(\cdot))(x)}{\rho(x)}\right|}\,
		 \leqslant  \kappa_{\infty}  \E{\sup_{x \in \bR^d} \left|\frac{v(x)-Pv^n_M(x)}{\rho(x)}\right|} = \kappa_{\infty} e_{\infty,n}.\,
	\end{split}
\end{equation*}

\paragraph{\it Error $\mathcal{E}_{\infty,3}$:}
We have assumed that $v$ is $C^2$. Then, by using the second inequality in Proposition \ref{prop:properties-P} and the boundedness of $v$, we have
\begin{align*}
 \sup_{x \in \bR^d} \left|\frac{Pv(x)-v(x)}{\rho(x)}\right| & \leqslant \sup_{x \in \Box} \left|\frac{Pv(x)-v(x)}{\rho(x)}\right| + \sup_{x \in \bR^d \setminus \Box} \left|\frac{Pv(x)-v(x)}{\rho(x)}\right|\\
 & \leqslant C\delta^2+ \frac{C}{\inf_{x \in \bR^d \setminus \Box}  \rho(x)}
\end{align*}
where $C$ does not depend on $\Pi$ since $\nabla^2 v$ is assumed to be bounded on $\bR^d$. Since $0 \in \Box$, we also have
$ \inf_{x \in \bR^d \setminus \Box}  \rho(x) = \inf_{x \in \partial \Box} \rho(x)$.

\paragraph{\it Error $e_{\infty,n}$:} 
Now we just have to collect previous estimates to get
\begin{align*}
 e_{\infty,n+1} \leqslant & C\sup_{x \in \bR^d} \left|\frac{P\rho(x)}{\rho(x)} \right| \inf_{z \in \Pi} \frac{\log({1+}N)(1+|z|)}{\sqrt{M_z} \rho(z)}+\kappa_{\infty} e_{\infty,n}+C\delta^2+  \frac{C}{\inf_{x \in \partial \Box}  \rho(x)}
\end{align*}
which gives us 
\begin{align*}
 e_{\infty,n} \leqslant &\left(\sum_{k=0}^{n-1} \kappa_{\infty}^k\right) \left(C\sup_{x \in \bR^d} \left|\frac{P\rho(x)}{\rho(x)} \right| \inf_{z \in \Pi} \frac{\log({1+}N)(1+|z|)}{\sqrt{M_z} \rho(z)}+C\delta^2+\frac{C}{\inf_{x \in \partial \Box}  \rho(x)}\right) +\kappa_{\infty}^n e_{\infty,0}.
\end{align*}
This last inequality is the one we wanted to prove.
\qed

\def\cprime{$'$}


\end{document}